\documentclass[10pt,leqno,twoside]{amsart}

\usepackage{enumerate}

\usepackage{amssymb}
\usepackage[english]{babel}
\usepackage{amssymb,amsthm,amsmath,eucal,mathrsfs}
\usepackage{bm}

\usepackage{amsmath}
\usepackage{amsfonts}
\usepackage{amsmath}
\usepackage{amssymb}
\usepackage{graphicx}
\usepackage{lscape}
\usepackage{amstext}
\usepackage{amsthm}
\usepackage{color}
\usepackage{float}
\usepackage{mathrsfs}
\usepackage{epsfig}
\usepackage{url}
\usepackage{fancyhdr}
\usepackage{pspicture}
\usepackage{graphicx}
\usepackage{cite}

\usepackage{amsmath,verbatim}

\usepackage{amsthm}

\usepackage{amssymb}

\usepackage{amsfonts}

\usepackage{dsfont}

\usepackage{hyperref}

\newtheorem{theorem}{Theorem}[section]

\newtheorem{proposition}[theorem]{Proposition}

\newtheorem{Hypothesis}{Hypothesis}

\theoremstyle{definition}

\newtheorem{remark}[theorem]{Remark}

\numberwithin{equation}{section}

\newcommand{\dist}{\mathrm{dist}}      % distance

\newcommand{\diam}{\mathrm{diam}}      % diameter

\renewcommand{\Re}{{\ensuremath{\mathrm{Re\,}}}} %Realteil nicht als fraktur

\renewcommand{\div}{\mathrm{div}\,}    %div anstatt geteilt

\newcommand\restr[2]{{% we make the whole thing an ordinary symbol
  \left.\kern-\nulldelimiterspace % automatically resize the bar with \right
  #1 % the function
  \vphantom{\big|} % pretend it's a little taller at normal size
  \right|_{#2} % this is the delimiter
  }}

\usepackage{eucal}

\title{Optical inversion using plasmonic contrast agents
}

\author[Cao, Ghandriche and Sini]{Xinlin Cao$^{\dag}$, Ahcene Ghandriche $^{**}$ and Mourad Sini$^{\ddag}$}
\thanks{$^{\dag}$Departement of Applied Mathematics, The Hong Kong Polytechnic University, Hong Kong SAR. Email: xinlin.cao@polyu.edu.hk}
\thanks{$^{**}$ Nanjing Center for Applied Mathematics, Nanjing, 211135, People's Republic of China. Email: gh.hsen@njcam.org.cn}
\thanks{$^{\ddag}$ RICAM, Austrian Academy of Sciences, Altenbergerstrasse 69, A-4040, Linz, Austria. Email: mourad.sini@oeaw.ac.at. This author is partially supported by the Austrian Science Fund (FWF): P 30756-NBL and P 32660}

\begin{document}

\begin{abstract}

We describe a new method to reconstruct the permittivity distribution, of an object to image, from the remotely measured electromagnetic field. We propose to use the remote fields measured before and after injecting locally in the medium plasmonic nano-particles. Such a technique is known in the framework of imaging using contrast agents where, in optical imaging, the nano-particles play the role of these contrast agents. The plasmonic nano-particles are known to enjoy resonant effects, as enhancing the applied incident field, while excited at certain particular frequencies called plasmonic resonances.
These resonant frequencies encode the values of the unknown permittivity at the location of the injected nano-particles. 
The imaging methods we propose mainly use this resonant effect. 
We show that the imaging functional build up from contrasting the fields before and after injecting the nano-particles, measured at one single back-scattered direction, and in an explicit band of incident frequencies, reaches its maximum values, in terms of the incident frequency, precisely at the mentioned plasmonic resonances. 
Such a behavior allows us to recover these plasmonic resonances from which we recover the point-wise values of the permittivity distribution.
\bigskip

In this work, we describe the method and provide the mathematical justification of this resonant effect and its use for the optical inversion using plasmonic nano-particles as contrast agents.

%{\color{red}{Please emphasize that we can use one single measurement (back scattering direction) in the main theorem to get the reconstruction result!}}

%This paper aims to reconstruct the permittivity function $\epsilon_{0}(\cdot)$, for the electromagnetic problem, of a bounded medium $\Omega$. Estimate has been made on the dominant term related to the contrast between the measured far-fields, before and after injecting plasmonic nano-particles inside $\Omega$. It's proved that the plasmonic resonances can be expressed in terms of the eigenvalues of the Magnetization operator and both the permittivity of the injected plasmonic nano-particles, given through the Lorentz model, and the unknown permittivity of the background, i.e. $\epsilon_{0}(\cdot)$. An algorithm for reconstructing the permittivity function is also explained.  
\end{abstract}


%\end{frontmatter}
\date{\today}

\allowdisplaybreaks

\maketitle

\section{Introduction and statement of the results}

In quantitative optical imaging, the goal is to estimate the optical properties, as the permittivity function, from the optical signal response of the object to image after excitation with coherent or incoherent incident electromagnetic fields, see \cite{Arridge, HAD, S-U-2003, Habib-book, GHA}. The optical response is measured at receivers located away from the object to image. Such inversion requires multiple emitters-receivers, i.e. many measurements, and it is known to be highly unstable \cite{Isakov} and hence one is unable to recover low contrasting permittivity functions using such remote measurements.  To overcome such difficulties, it is proposed in the engineering literature, see \cite{BBC, QCF, Ahcene-Mourad-ICCM, Anderson-et-all, Ilovitsh, Quaia, P-P-B:2015, Li-Chen}, to perturb the medium with small-scaled inhomogeneities to create such missing contrast to render the imaging more accurate. In optics, such contrast agents are given by nano-particles which enjoy appropriate contrasting properties. We have two classes of such nano-particles: dielectric and plasmonic nano-particles. The dielectric nano-particles are highly localized as they are nano-scaled and have high contrast permittivity.  Under these scales, we can choose the incident frequency so that we excite the dielectric resonances which are related to the eigenvalues of the vector Newtonian operator, see \cite{ammari2023, cao2023} for a mathematical justification. Moreover, in practice, optically induced waves by dielectric nano-particles have been examined and extensively studied, see for instance \cite{TS2018, KMBKL, zograf2020all, zograf2021all}. The main feature of the plasmonic nano-particles is that they enjoy negative values of the real part of their permittivity if we choose incident frequencies close to the plasmonic frequencies of the nano-particle. With such negative permittivity, we can excite the plasmonic resonances which are related to the eigenvalues of the Magnetization operator, see \cite{AMMARI2019242, AG-MS-Maxwell}. These plasmonic nano-particles show an ability to manipulate light at the nano-scaled size due to their support for these resonant optical modes, see \cite{CP, maier2007plasmonics, FZS, liberman2010optical, N-H-Book, hao2004electromagnetic, NobleMetals, zeman1987, baffou2009} for an extensive description and studies of their properties. Such unique properties are the reasons behind the increase of interest in the study of plasmonic nano-particles, in particular those made of gold or silver, as demonstrated in \cite{N-H-Book}. These optical properties of plasmonic nano-particles have been utilized in numerous applications, including bio-sensors, thermo-therapy, and solar cells, etc. To learn more on these two types of nano-particles about their common as well as different properties, the interested reader can see the references \cite{zograf2020all, zograf2021all, KMBKL}.\\
To describe the material properties of these nano-particles, we use the Lorentz model where the permeability $\mu$ is kept constant as the one of the homogeneous background while the permittivity has the form:
\begin{equation}\label{Lorentz-model}
\epsilon_{p}(\omega) \, = \, \epsilon_\infty \left( 1+\frac{\omega^2_p}{\omega^2_0-\omega^2-i \gamma \omega} \right)
\end{equation} 
where $\omega_p$ is the electric plasma frequency, $\omega_0$ is the undamped frequency, and $\gamma$ is the electric damping frequency which is assumed to be small, i.e., $0 \leq \gamma \ll 1$, and its order of smallness will be discussed later, see \eqref{gammaah}. We refer \cite[Formula (4)]{PhysRevA.82.055802}  or \cite[Formula (1.3)]{engheta2006metamaterials} for the derivation of this model. It is observed that if we choose the incident frequency $\omega$ so that $\omega^2$ is larger than $\omega_0^2$, then the real part becomes negative. For such choices of the incident frequency, the nano-particle behaves as a plasmonic nano-particle. In the current work, we focus on the use of these properties for quantitative imaging. As mentioned before, these nano-particles are used as contrast agents. Few scenarios and assumptions on the distribution of the contrast agents can be considered. These scenarios are of course not exhaustive as the reality might be more complicated. Nevertheless, we state the following classes of distributions  under which we can perform rigorous analysis. 

\begin{enumerate}
\item The contrast agents are injected in isolation, i.e. they are injected one after another. In addition, we need to assume them well separated to insure weak multiple scattering between them.
\bigskip

\item They are injected as a cluster, i.e. all-at-once. In this case we have full strong multiple scattering between them. Here also, we can handle the following situations.
\bigskip

\begin{enumerate}

\item The distribution is regular, i.e. periodic or following a given density of distribution. In this case, we can expect deterministic estimations.
\bigskip

\item The distribution is random. In this case, we aim for probabilistic estimations.
\end{enumerate}
\end{enumerate}
\bigskip

The approach we propose to analyse these families of imaging modalities using contrast agents can be summarized as follows. 
Contrasting the measured fields collected before and after injecting the small-agents, we propose the following solutions.
\begin{enumerate}
\item[]
\item Case when we inject the contrast agents one after another, under the weak scattering assumption.
\begin{enumerate}
\item[] 
\item  We can recover the  {\it{related resonances}}. From these resonances, we derive the values of the 'high order' coefficient (i.e. the mass density in acoustics). This idea works in the time-harmonic regimes.
\item[] 
\item We can recover the internal values of {\it{the travel time function}}. Using the Eikonal equation, we recover the speed of propagation. This idea works in time-domain regimes.
\item[] 
\item In addition, we can recover the internal values of the {\it{total fields}} generated only by the background medium. This allows to recover the lower order coefficients (as the bulk modulus in acoustics). This idea works for both the time-harmonic and the time-domain regimes. 
\end{enumerate}
\bigskip

These ideas were applied to the time-harmonic as well as the time domain imaging for acoustics, i.e. ultrasound imaging using bubble as contrast agents, see \cite{dabrowski2021, Soumen-Minnaert, SSW}.
\bigskip

\item  Case when we inject the contrast agents {\it{all-at-once}}, as a cluster.
\bigskip

\begin{enumerate}

 \item using the resonant character of the contrast-agents,
 we can {\it{linearize}} the measured boundary-map, i.e. the Dirichlet-Neumann map.
\bigskip

\item  then we solve the linearized inverse problem.
\end{enumerate}
\bigskip

This idea was applied in \cite{ghandriche2023calderon} to the Calderón problem using resonant perturbations modeling the acoustic imaging with droplets as contrast agents.
\end{enumerate}
\bigskip

\noindent So far, this approach was applied to acoustic waves based imaging, i.e. for the ultrasound imaging modality using bubbles or droplets as contrast agents and also photo-acoustic imaging using nano-particles, see \cite{AG-MS-Maxwell, AG-MS-Simultaneous}. Our goal in the current work is to extend this approach to  the electromagnetic waves based imaging as the case of quantitative optical imaging using nano-particles. To describe the mathematical model behind the optical experiment, we set the electric field $E(\cdot)$ to be solution of the following system
\begin{equation}
\label{eq:electromagnetic_scattering}
\left\{
\begin{array}{lll}
\nabla \times \nabla \times \left( E \right) \, - \, \omega^2 \; \varepsilon \; \mu \; E \,  =  \, 0, \quad \; \qquad \; \, \text{in} \, \, \mathbb{R}^{3}, \\ 
 E \, :=  \, E^{s} \, + \, E^{i}, \qquad \qquad \qquad \qquad \qquad  \mbox{ in } \mathbb{R}^{3},\\
\underset{\left\vert x \right\vert \rightarrow + \infty}{\lim} \left\vert x \right\vert \; \left(\nabla \times \left( E^{s}(x) \right) \times x \, - \, E^{s}(x) \right)  = 0,
\end{array}
\right.
\end{equation}
with the last condition known as the Silver-M\"{u}ller radiation conditions. Here, $\omega$ is the incident frequency, $\mu$ is the permeability parameter, which will be taken to be constant in the whole space $\mathbb{R}^{3}$, and $\varepsilon(\cdot)$ is the permittivity function 
defined as 
\begin{equation}\label{DefPerFct}
\varepsilon(x) := \begin{cases} 
\epsilon_{\infty} & in \quad \mathbb{R}^{3} \setminus \Omega, \\
\epsilon_{0}(x) & in \quad \Omega \setminus D, \\
\epsilon_{p}(\omega) & in \quad D,
\end{cases}
\end{equation}
with $\epsilon_{\infty}$ being a positive constant used to represent the permittivity of the background (outside $\Omega$), $\epsilon_{p}(\omega)$ is given by $(\ref{Lorentz-model})$, and the permittivity $\epsilon_0(\cdot)$ is variable and it is supposed to be smooth of class $\mathcal{C}^{1}$ inside $\Omega$. Besides, we assume that $Im\left( \epsilon_{0}(\cdot) \right)$ is small such that 
\begin{equation}\label{ImEpsgamma}
    \left\Vert Im\left( \epsilon_{0}(\cdot) \right) \right\Vert_{\mathbb{L}^{\infty}(\Omega)} \; = \; \mathcal{O}\left( \gamma \right), 
\end{equation}
where $\gamma$ is the electric damping frequency parameter of the Lorentz model, see $(\ref{Lorentz-model})$. The domain $D$ is given as a collection of $\aleph$ connected and $\mathcal{C}^{2}$-smooth nano-particles $D_{i}$'s, i.e. $D \, := \, \overset{\aleph}{\underset{i=1}\cup} D_{i}$. In addition $D \subset \Omega$, where $\Omega$ is a bounded and $\mathcal{C}^{2}$-smooth domain in $\mathbb{R}^{3}$. Related to the permittivity function given by $(\ref{DefPerFct})$, we set the index of refraction $\bm{n}$, in $\mathbb{R}^{3}$, given by 
\begin{align}\label{Def-Index-Ref}
  \bm{n} := \begin{cases}
      \sqrt{\epsilon_{p} \, \mu} & \text{in $D$} \\
      \bm{n}_{0} & \text{in $\mathbb{R}^{3} \setminus D$}
    \end{cases} 
    \qquad \text{and} \qquad \bm{n}_{0} := \begin{cases}
      \sqrt{\epsilon_{0}(\cdot) \, \mu} & \text{in $\Omega$} \\
      \sqrt{\epsilon_{\infty} \, \mu} & \text{in $\mathbb{R}^{3} \setminus \Omega$}
    \end{cases}. 
\end{align} 
Each nano-particle $D_{j}$, for $1 \leq j \leq \aleph$, is taken of the form $D_{j} \, :=a \; B_{j} \, +z_{j}$ where $z_{j}$ models its location and $a$ signifies its relative radius with $B_{j}$ as $\mathcal{C}^2$-smooth domain of maximum radius $1$ such that $B_{j} \subset B(0,1)$, where $B(0,1)$ is the unit ball centered at the origin. The parameter $a$ is defined by 
\begin{equation*}
    a \, := \, \underset{1 \leq j \leq \aleph}{\max} \; \diam\left( D_{j} \right),
\end{equation*}
and we denote $d$ as the minimal distance between any two of the distributed nano-particles, i.e. 
\begin{equation*}
    d \, = \, \underset{1 \leq i, j \leq \aleph \atop i \neq j}{\min} \; \dist\left( D_{i}, D_{j} \right). 
\end{equation*}
The parameters $d$ and $\aleph$, are linked to the parameter $a$ through the following behaviors
\begin{equation}\label{datNas}
    d \, \sim \, a^{t} \quad \text{and} \quad \aleph \, \sim \, \left[ a^{-s} \right], \quad \text{with} \quad 0 \, < \, t, s \, < \, 1, 
\end{equation}
where $\left[ \cdot \right]$ stands for the entire part function. Moreover, for short notations, we denote $d_{mj}$ as the distance between $z_{m}$ and $z_{j}$, where $z_{m}$ (respectively, $z_{j}$) is the center of $D_{m}$ (respectively, $D_{j}$), i.e.
\begin{equation*}\label{dmj}
    d_{mj} \, := \, \left\vert z_{m} \, - \, z_{j} \right\vert, \quad \text{with} \quad 1\leq m, j\leq \aleph,\quad m\neq j. 
\end{equation*}The problem $(\ref{eq:electromagnetic_scattering})$ is well-posed in appropriate Sobolev spaces, see for example \cite[Theorem 2.1]{kirsch2004factorization}.   We assume that the incident field $E^{Inc}(\cdot, \cdot, \cdot)$ is given by
 \begin{equation}\label{DefEinc}
    E^{Inc}\left( x, \theta, q \right) \, = \, \left( \theta \times q \right) \, e^{i \, k \, x \cdot \theta}, \quad \text{for} \quad \theta , q \in \mathbb{S}^{2} \quad \text{and} \quad x \in \mathbb{R}^{3}, 
\end{equation}
with $\mathbb{S}^{2}$ being the unit sphere, $\theta$ is the incident direction vector and $q$ is the polarization vector such that $q \cdot \theta \, = \, 0$. The incident field $E^{Inc}\left( \cdot, \cdot, \cdot \right)$ is the solution to
 \begin{equation*}\label{EquaDefEinc}
    \underset{x}{\nabla} \times \underset{x}{\nabla} \times \left( E^{Inc}\left( x, \theta, q \right) \right) \, - \, k^{2} \, E^{Inc}\left( x, \theta, q \right) \; = \; 0, \quad x \in \mathbb{R}^{3}, 
\end{equation*}
where the wave number $k$ is given by the positive constant $k \, = \, \omega \, \sqrt{\epsilon_{\infty} \, \mu}$. Besides, the scattered wave $E^{s}\left( \cdot, \cdot, \cdot \right)$ has the following asymptotic behavior,
\begin{equation*}
    E^{s}\left( \hat{x}, \theta, q \right) \, = \, \frac{e^{i \, k \, \left\vert x \right\vert}}{\left\vert x \right\vert} \, \left( E^{\infty}\left( \hat{x}, \theta, q \right) \, + \, \mathcal{O}\left( \frac{1}{\left\vert x \right\vert} \right) \right), \quad \left\vert x \right\vert \rightarrow + \infty,
\end{equation*}
where $E^{\infty}\left( \hat{x}, \theta, q \right)$ is the corresponding electromagnetic far-field pattern of $(\ref{eq:electromagnetic_scattering})$ in the propagation direction $\hat{x}$. 
\medskip
\newline
Motivated by the use of integral equations to represent solutions to Maxwell's equations, and for future use, we introduce the Newtonian operator $N^{k}_{B}(\cdot)$ and the Magnetization operator $\nabla M^{k}_{B}(\cdot)$, both acting on vector fields
\begin{equation}\label{DefNDefMk}
N^{k}_{B}(F)(x) \, := \, \int_{B} \Phi_{k}(x,y) \, F(y)dy \quad \text{and} \quad  \nabla M^{k}_{B}(F)(x) \, := \, \nabla \int_{B}\underset{y}{\nabla}\Phi_{k}(x,y) \cdot F(y)dy,
\end{equation} 
where 
\begin{equation}\label{FSHE}
    \Phi_{k}(x,y) \; := \; \frac{e^{i \, k \, \left\vert x-y \right\vert}}{4\pi \left\vert x-y\right\vert}, \quad x \neq y, 
\end{equation}
is the fundamental solution to the Helmholtz equation in the entire space. Particularly, for $k = 0$ we obtain
\begin{equation}\label{DefNDefM}
N_{B}(F)(x) \, := \, \int_{B} \frac{1}{4\pi \vert x-y\vert} F(y)dy \quad \text{and} \quad \nabla M_{B}(F)(x) \, := \, \nabla \int_{B} \underset{y}{\nabla}\left( \frac{1}{4\pi \vert x-y\vert} \right) \cdot F(y)dy.
\end{equation} 
We will observe later that the introduced operators, see $(\ref{DefNDefMk})$, appear after taking convolution of vector fields with the Green's kernel $G_{k}(\cdot, \cdot)$ associated with the problem $(\ref{eq:electromagnetic_scattering})$. More precisely, $G_{k}(\cdot, \cdot)$ is the solution, in the distributional sense, to     
\begin{equation}\label{Eq0401}
    \underset{y}{\nabla} \times \underset{y}{\nabla} \times G_{k}(x, y) \, - \, \omega^{2} \, \bm{n}_{0}^{2}(y)  \, G_{k}(x,y) \, = \, \delta(x,y) \, I, \quad x, y \; \in \mathbb{R}^{3},  
\end{equation}
such that each column of $G_{k}(x,\cdot)$ satisfies the outgoing radiation condition 
\begin{equation*}
    \underset{\left\vert x \right\vert \rightarrow + \infty}{\lim} \left\vert x \right\vert \; \left( \underset{y}{\nabla} \times G_{k}(x, y) \times \frac{x}{\left\vert x \right\vert} \, - \, i \, k \, G_{k}(x, y) \right) \; = \; 0.
\end{equation*}
When dealing with vector fields, it is crucial to recall the Helmholtz decomposition for $\mathbb{L}^{2}(B)$-space given by 
\begin{equation}\label{HelmholtzDecomposition}
\left( \mathbb{L}^2(B)\right)^{3} \, = \, \mathbb{H}_{0}(\div = 0)(B) \oplus \mathbb{H}_{0}(Curl = 0)(B) \oplus \nabla \mathcal{H}armonic(B)
\end{equation}
where  
\begin{eqnarray*}
 \mathbb{H}_{0}(\div = 0)(B) \, &:=& \, \{u \in \left( \mathbb{L}^{2}(B) \right)^{3} \big| \nabla \cdot \left( u \right) \, = \, 0, \nu \cdot u=0 \mbox{ on } \partial B\},   \\
 \mathbb{H}_{0}(Curl = 0)(B) \, &:=& \, \{u \in \left( \mathbb{L}^{2}(B) \right)^{3}\big| \nabla \times u \, = \, 0, \nu \times u=0 \mbox{ on } \partial B\}
\end{eqnarray*}
and
\begin{equation*}
\nabla \mathcal{H}armonic(B) \, := \, \{u \in \left( \mathbb{L}^{2}(B) \right)^{3} \big|u=\nabla \phi,~~ \Delta \phi =0 \mbox{ in } B\}.
\end{equation*}
\medskip
\newline
Of all possible $\left( \mathbb{L}^2(B) \right)^{3}$-space decompositions, $(\ref{HelmholtzDecomposition})$ is the most natural one, as we know that $N{|_{\mathbb{H}_{0}(\div = 0)(B)}}$ and $N{|_{\mathbb{H}_{0}(Curl = 0)(B)}}$ generate a complete orthonormal bases $\left(\lambda_{n}^{(1)}(B); e_{n}^{(1)} \right)_{n \in \mathbb{N}}$ and $\left(\lambda_{n}^{(2)}(B); e_{n}^{(2)} \right)_{n \in \mathbb{N}}$ respectively. In addition, it is known that $\nabla M : \, \nabla \mathcal{H}armonic(B) \rightarrow \nabla \mathcal{H}armonic(B)$ has a complete basis $\left(\lambda_{n}^{(3)}(B); e_{n}^{(3)} \right)_{n \in \mathbb{N}}$, see \cite[Proposition 5.1]{AG-MS-Maxwell}. For an in-depth study of the properties of the Magnetization operator, the reader can refer to \cite{AhnDyaRae99, friedman1980mathematical, friedman1981mathematical, 10.2307/2008286, Dyakin-Rayevskii} and \cite{Raevskii1994}.  
\bigskip
\newline 
In the sequel, we fix an $n_{0} \in \mathbb{N}$ and we let the used incident frequency $\omega$ to be of the form 
\begin{equation*}
    \omega^{2} \, =  \, \omega^{2}_{P_{\ell},n_{0},j} \, \pm \, C \, a^{h}, \quad h \in (0,1) \quad \text{and} \quad 1 \leq j \leq \aleph, 
\end{equation*}
where $\omega^{2}_{P_{\ell},n_{0},j}$ is the plasmonic resonance related to the eigenvalue $\lambda_{n_{0}}^{(3)}(B)$, see $(\ref{PlasmonicResonace})$, and $C$ is a constant independent on the parameter $a$. 
\medskip
\newline 
We are now in a position to state the primary outcome of this work.
\medskip
\begin{theorem}\label{ThmResult} Under the regularity assumptions on $\Omega, D, \epsilon$ and  $\mu$ described above, we have the following expansions.
\begin{enumerate}
\item[] 
\item For the scattered fields, with $x$ is away from $D$,
\begin{eqnarray}\label{Es-Vs-intro}
\nonumber
\left( E^{s} \, - \, V^{s}\right)(x, \theta)  \, & = & - \,  \mu \, a^{3} \, \omega^{2} \, \sum_{j=1}^{\aleph}  \, \frac{\epsilon_{0}(z_{j}) \, \left(\epsilon_{0}(z_{j}) \, - \, \Lambda_{n_{0},j}\left( \omega \right) \right)}{\lambda_{n_{0}}^{(3)}(B) \, \Lambda_{n_{0},j}\left( \omega \right)} \, \langle V(z_{j}, \theta, q), \int_{B}  e_{n_{0}}^{(3)}(y)  \, dy \rangle \, G_{k}(x,z_{j}) \cdot \int_{B}  e_{n_{0}}^{(3)}(y)  \, dy  \\ &+& \mathcal{O}\left( a^{\min\left( (3-s), (6-3t-2h-\frac{3s}{2}), (4 - h - s) \right)}\right).
\end{eqnarray}
\item[] 
\item For the far-fields
\begin{eqnarray}\label{Es-Vs-intro-far-field}
\nonumber
&& \langle \left( E^{\infty} \, - \, V^{\infty}\right)(\hat{x}) , \left( \hat{x} \times q \right) \rangle \\ \nonumber & = & \, \frac{ \mu \, a^{3}}{4 \, \pi}  \, \sum_{j=1}^{\aleph}  \, \frac{\omega^{2}_{P_{\ell},n_{0},j} \, \epsilon_{0}(z_{j}) \, \left(\epsilon_{0}(z_{j}) \, - \, \Lambda_{n_{0},j}\left( \omega_{P_{\ell},n_{0},j} \right) \right)}{\lambda_{n_{0}}^{(3)}(B) \, \Lambda_{n_{0},j}\left( \omega \right)}  \, \langle V(z_{j},  \theta, q), \int_{B}  e_{n_{0}}^{(3)}(y)  \, dy \rangle \, \langle V\left(z_{j}, - \hat{x}, q \right), \int_{B}  e_{n_{0}}^{(3)}(y)  \, dy  \rangle  \\ && \qquad \qquad + \mathcal{O}\left( a^{\min\left( (3-s), (6-3t-2h-\frac{3s}{2}), (4 - h - s) \right)}\right). 
\end{eqnarray}
\end{enumerate}
In particular, the far-field, in the back-scattered direction $\hat{x}=-\theta$, has the approximation 
\begin{eqnarray}\label{back-scattered-far-field}
\nonumber
&& \langle \left( E^{\infty} \, - \, V^{\infty}\right)(- \, \theta) , \left( \theta \times q \right) \rangle = \\ \nonumber  && \, - \frac{\mu \, a^{3}}{4 \, \pi}  \, \sum_{j=1}^{\aleph}  \, \frac{\omega^{2}_{P_{\ell},n_{0},j} \, \epsilon_{0}(z_{j}) \, \left(\epsilon_{0}(z_{j}) \, - \, \Lambda_{n_{0},j}\left( \omega_{P_{\ell},n_{0},j} \right) \right)}{\lambda_{n_{0}}^{(3)}(B) \, \Lambda_{n_{0},j}\left( \omega \right)} \, \left(  \langle V\left(z_{j},  \theta, q \right), \int_{B}  e_{n_{0}}^{(3)}(y)  \, dy \rangle \right)^{2} \\ &+& \mathcal{O}\left( a^{\min\left( (3-s), (6-3t-2h-\frac{3s}{2}), (4 - h - s) \right)}\right). 
\end{eqnarray}
The above expansions are valid under the conditions $0 \leq h, t, s < 1$ such that 
\begin{equation}\label{3hts0}
    3 - h - 3t - s \, > \, 0,
\end{equation}
with $\hat{x}, q \, \in \mathbb{S}^{2}$ and $\Lambda_{n_{0},j}\left( \cdot \right)$ is the dispersion function given by 
\begin{equation}\label{Thm-Dis-Eq}
    \Lambda_{n_{0},j}\left( \omega \right) \, := \, \left( \epsilon_{0}(z_{j}) \, - \, \lambda_{n_{0}}^{(3)} \, \left( \epsilon_{0}(z_{j}) \, - \, \epsilon_{p}(\omega) \right) \right), 
\end{equation} 
where $\epsilon_{p}(\omega)$ is the Lorentz model for the permittivity given by $(\ref{Lorentz-model})$. 
\bigskip

\noindent The field $V$ (as its related scattered and far-field pattern) is the solution of the problem $(\ref{eq:electromagnetic_scattering})$ in the absence of the nano-particles.
\end{theorem}
\medskip
The expansion $(\ref{back-scattered-far-field})$ describes the field generated by the collection of nano-particles neglecting the mutual interaction between them, i.e. the Born approximation. This Born approximation is valid under the condition $(\ref{3hts0})$. Avoiding the mutual interaction between the nano-particles makes it easier to state an imaging functional for the optical inversion. The details are provided in Section \ref{SectionII}. The condition $(\ref{3hts0})$ can be relaxed. In this case the approximation $(\ref{back-scattered-far-field})$ becomes the Foldy-approximation (instead of the Born approximation) which involves the multiples scattering effects. Recall that the number of injected nano-particles is $\aleph \sim a^{-s}$, where $s$ satisfies $(\ref{3hts0})$. Therefore this set of nano-particles is dense in $\Omega$. Indeed, as $\aleph \sim d^{-3} \sim a^{-3t}$, recalling that $d \sim a^{t}$, then $s=3t$ and hence $(\ref{3hts0})$ becomes $t < \frac{1}{2} - \frac{h}{6}$. Since $h < 1$, then we need the condition $t < \frac{1}{3}$. Therefore the nano-particles can be distributed with a minimum distance between them of the order at least $d \ll a^{\frac{1}{3}}$.  
\medskip
\newline 
The remaining part of the manuscript is divided as follows. In Section \ref{SectionII}, we provide a detailed explanation of how the results in Theorem \ref{ThmResult} can be used to propose an algorithm in optical imaging to reconstruct the permittivity function $\epsilon_{0}(\cdot)$, in the bounded medium $\Omega$. In Section \ref{SecPrfThm} we prove Theorem \ref{ThmResult} while postponing the justification for the Mixed Electromagnetic Reciprocity Relation and an a-priori estimate related to the total electric field to the next section. In Section \ref{SectionIV}, which will be given as an appendix, we show the Mixed Electromagnetic Reciprocity Relation in Subsection \ref{P21}, and  derive an a-priori estimate of the electric total field in Subsection \ref{Subsection42}. Furthermore, Subsection \ref{SubSecIIEqua242} and Subsection \ref{SubsectionGT=G} are included to finish Subsection \ref{P21}.   

\section{Application to the optical imaging using plasmonic contrast agents}\label{SectionII}
We develop quantitative an imaging procedure that can reconstruct the permittivity function $\epsilon_{0}(\cdot)$ within the bounded medium $\Omega$, to be imaged, using the expansion given in Theorem \ref{ThmResult}, formula $(\ref{back-scattered-far-field})$. To use those expansions, the following condition is needed and we present it as a hypothesis to maintain greater generality.
\begin{Hypothesis}\label{Hyp}
There exist $n_{0} \in \mathbb{N}$ and $\theta, q \in \mathbb{S}^{2}$ such that 
\begin{equation}\label{CdtHyp}
    \langle V\left(\cdot, \theta, q \right), \int_{B} e^{(3)}_{n_{0}}(y) \, dy \rangle \; \neq \; 0, 
\end{equation}
where $V\left(\cdot, \theta, q \right)$ is solution to $(\ref{eq:electromagnetic_scattering})$, with\footnote{The solution $V\left(\cdot, \theta, q \right)$ corresponds to the unperturbed medium case, i.e. before injecting the nano-particle $D$.} $\varepsilon(\cdot) \, = \, \epsilon_{0}(\cdot)$ in $\mathbb{R}^{3}$.
\end{Hypothesis}
\medskip
The remark below is suggested by the hypothesis above.
\begin{remark}  Three remarks are in order.
\begin{enumerate}
    \item[]
    \item Similarly, condition $(\ref{CdtHyp})$ implies that the permittivity function $\epsilon_{0}(\cdot)$ cannot be reconstructed near points where the total field $V\left(\cdot, \theta, q \right)$ is vanishing. This is quite reasonable since the permittivity is attached to the total field V, see the equation $(\ref{eq:electromagnetic_scattering})$, .
    \item[]
    \item In the case of a unit ball, i.e. $B \, = \, B(0,1)$, an explicit computation of $\int_{B} e^{(3)}_{1}(y) \, dy$ has been given in \cite[Section 4.5.3]{Gh-Thesis}. 
    \item[]
    \item In Hypothesis \ref{Hyp}, we used the notation
    \begin{equation*}
    \langle V\left(\cdot, \theta, q \right), \int_{B} e^{(3)}_{n_{0}}(y) \, dy \rangle \; = \; \sum_{m} \, \langle V\left(\cdot, \theta, q \right), \int_{B} e^{(3)}_{n_{0},m}(y) \, dy \rangle, 
    \end{equation*}
where $e^{(3)}_{n_{0},m}(\cdot)$ are such that
\begin{equation*}
    \nabla M_{B} \left( e^{(3)}_{n_{0},m} \right) \, = \, \lambda^{(3)}_{n_{0}}\left( B \right) \, e^{(3)}_{n_{0},m}, \quad \text{in} \; B.
\end{equation*}

\end{enumerate}

\end{remark}
\bigskip

Under the above hypothesis, we start by recalling $(\ref{back-scattered-far-field})$ that 
\begin{eqnarray*}
\nonumber
&& \langle \left( E^{\infty} \, - \, V^{\infty}\right)(- \, \theta, \theta, q, \omega) , \left( \theta \times q \right) \rangle \\ &=& - \frac{\mu \, a^{3}}{4 \, \pi}  \, \sum_{j=1}^{\aleph}  \, \frac{\omega^{2}_{P_{\ell},n_{0},j} \, \epsilon_{0}(z_{j}) \, \left(\epsilon_{0}(z_{j}) \, - \, \Lambda_{n_{0},j}\left( \omega_{P_{\ell},n_{0},j} \right) \right)}{\lambda_{n_{0}}^{(3)}(B) \, \Lambda_{n_{0},j}\left( \omega \right)} \, \left(  \langle V\left(z_{j},  \theta, q \right), \int_{B}  e_{n_{0}}^{(3)}(y)  \, dy \rangle \right)^{2} \\ &+& \mathcal{O}\left( a^{\min\left( (3-s), (6-3t-2h-\frac{3s}{2}), (4 - h - s) \right)}\right), 
\end{eqnarray*}
which, by keeping only its dominant term, becomes
\begin{equation}\label{C-F-F-B-S-D}
\langle \left( E^{\infty} \, - \, V^{\infty}\right)(- \, \theta, \theta, q, \omega) , \left( \theta \times q \right) \rangle  \, \simeq \,- \frac{\mu \, a^{3}}{4 \, \pi \, \lambda_{n_{0}}^{(3)}(B) }  \, \sum_{j=1}^{\aleph}  \, \mathcal{J}\left(\omega, z_{j} \right), 
\end{equation}
where the imaging functional $\mathcal{J}\left(\cdot, \cdot \right)$, depending on both the frequency parameter $\omega$ and the location points $\left\{ z_{j} \right\}_{j=1}^{\aleph}$, is given by  
\begin{equation}\label{ImgFct}
\mathcal{J}\left(\omega, z_{j} \right) \, := \, \frac{\omega^{2}_{P_{\ell},n_{0},j} \, \epsilon_{0}(z_{j}) \, \left(\epsilon_{0}(z_{j}) \, - \, \Lambda_{n_{0},j}\left( \omega_{P_{\ell},n_{0},j} \right) \right)}{ \Lambda_{n_{0},j}\left( \omega \right)} \, \left(  \langle V\left(z_{j},  \theta, q \right), \int_{B}  e_{n_{0}}^{(3)}(y)  \, dy \rangle \right)^{2},
\end{equation}
with $\Lambda_{n_{0},j}\left( \cdot \right)$ being the dispersion equation given by $(\ref{Thm-Dis-Eq})$. It is evident from the imaging functional expression, see  $(\ref{ImgFct})$, that the reconstruction of the permittivity function $\epsilon_{0}(\cdot)$ on the location points, i.e. $\left\{ \epsilon_{0}(z_{j}) \right\}_{j=1}^{\aleph}$, can be derived from the reconstruction of the imaging functional on the same location points, i.e. $\left\{ \mathcal{J}\left(\omega, z_{j} \right) \right\}^{\aleph}_{j=1}$, with 
\begin{equation*}\label{Freq-Int}
    \omega \, \in \, \mathbf{I} \, := \, \left( \omega_{0}\, , \, \sqrt{\omega_{0}^{2} \, + \, \frac{\omega_{p}^{2}}{\lambda^{(3)}_{n_{0}}(B)}} \right).
\end{equation*}
More details will be provided later in Subsection \ref{ProofThemSubSecII}. Furthermore, the contrast between the two far-fields $V^{\infty}(\cdot)$ and $E^{\infty}(\cdot)$, given by $(\ref{C-F-F-B-S-D})$, measured in the back-scattered direction, i.e. $\hat{x} \, = \, - \, \theta$, provides us with 
\begin{equation}\label{Mean-Img-Fct}
    \mathcal{F}(\omega, \aleph) \, := \, \sum_{j=1}^{\aleph}  \, \mathcal{J}\left(\omega, z_{j} \right), 
\end{equation}
up to a known multiplicative constant given by $-\dfrac{\mu \, a^{3}}{4 \, \pi \, \lambda_{n_{0}}^{(3)}(B) }$. Unfortunately, we cannot determine the $\aleph$ imaging functional $\mathcal{J}\left(\omega, z_{j} \right)$, for $1 \leq j \leq \aleph$, from the mean (sum) of the imaging functional $\mathcal{F}(\omega, \aleph)$, given by $(\ref{Mean-Img-Fct})$, as the problem is not uniquely solvable. Hence, the reconstruction of the permittivity function $\epsilon_{0}(\cdot)$ on the location points $\left\{ z_{j} \right\}_{j=1}^{\aleph}$ cannot be done, in a straightforward manner, as explained above. In order to solve this issue, we propose employing an iterative method that requires knowledge of $\left\{ \mathcal{F}(\omega, \ell) \right\}_{\ell = 1}^{\aleph}$,  given by $(\ref{Mean-Img-Fct})$. In this case, each imaging functional $\mathcal{J}\left(\omega, z_{j} \right)$ can be obtained as
\smallskip
\newline
\begin{eqnarray*} 
    \mathcal{J}\left(\omega, z_{1} \right) \, &=& \, \mathcal{F}(\omega, 1) \, = \, \text{Initial Reconstructed Term}, \\
    \mathcal{J}\left(\omega, z_{j} \right) \, &=& \, \mathcal{F}(\omega, j) \, - \, \mathcal{F}(\omega, j-1), \quad \text{for} \quad 2 \leq j \leq \aleph. 
\end{eqnarray*}
\smallskip
\newline
Based on the explanations mentioned above, our proposed imaging procedures follow as below.
\begin{enumerate}
    \item[] 
    \item Step 1)\label{STEPI}. Reconstructing $\mathcal{J}\left(\omega, z_{1} \right)$.
\begin{enumerate}
    \item[]
    \item \label{StepIa} Collect the far-field before injecting any plasmonic nano-particle inside $\Omega$, in the back-scattered direction at a single incident wave $\theta$, i.e. $V^{\infty}(- \, \theta, \theta, q, \omega)$, where $\omega \in \mathbf{I}$.
    \item[]
    \item \label{StepIb} Collect the far-field after injecting the first plasmonic nano-particle $D_{1}$, in the back-scattered direction at a single incident wave $\theta$, i.e. $E^{\infty}(- \, \theta, \theta, q, \omega)$, where $\omega \in \mathbf{I}$.
    \item[] 
\end{enumerate}
At this stage, since we have plugged only one single plasmonic nano-particle inside $\Omega$, the equation $(\ref{C-F-F-B-S-D})$ will collapse to the following single term
\begin{equation*}
\langle \left( E^{\infty} \, - \, V^{\infty}\right)(- \, \theta, \theta, q, \omega) , \left( \theta \times q \right) \rangle  \, \simeq \, -\frac{\mu \, a^{3}}{4 \, \pi \, \lambda_{n_{0}}^{(3)}(B) }  \, \mathcal{J}\left(\omega, z_{1} \right), 
\end{equation*}
where the left hand side is a known (measured) term, see for instance $(\ref{StepIa})$ and $(\ref{StepIb})$. In addition, as stated previously, the constant $-\dfrac{\mu \, a^{3}}{4 \, \pi \, \lambda_{n_{0}}^{(3)}(B) }$ appearing on the right hand side is known. Consequently, we deduce the reconstruction of the first imaging functional
\begin{equation}\label{Rec-J1}
    \omega \longrightarrow \mathcal{J}\left(\omega, z_{1} \right), \quad \text{with} \quad \omega \in \mathbf{I}.
\end{equation}
\item[]
\item Step 2)\label{STEPII}. Reconstructing $\epsilon_{0}(z_{1})$. \\ 
\newline 
As we have reconstructed the first imaging functional $\mathcal{J}\left(\omega, z_{1} \right)$, see for instance $(\ref{Rec-J1})$, and by recalling its analytical expression given by $(\ref{ImgFct})$,  
\begin{equation*}
\mathcal{J}\left(\omega, z_{1} \right) \, = \, \frac{\omega^{2}_{P_{\ell},n_{0},1} \, \epsilon_{0}(z_{1}) \, \left(\epsilon_{0}(z_{1}) \, - \, \Lambda_{n_{0},j}\left( \omega_{P_{\ell},n_{0},1} \right) \right) \, \left(  \langle V\left(z_{1},  \theta, q \right), \int_{B}  e_{n_{0}}^{(3)}(y)  \, dy \rangle \right)^{2}}{ \Lambda_{n_{0},1}\left( \omega \right)},
\end{equation*}
which can be rewritten, by introducing the proportional symbol $\boldsymbol{\alpha}$, like 
\begin{equation}\label{J1}
\mathcal{J}\left(\omega, z_{1} \right) \, \boldsymbol{\alpha} \, \frac{1}{ \Lambda_{n_{0},1}\left( \omega \right)}, \quad \omega \in \mathbf{I}, 
\end{equation}
since, with respect to the parameter $a$, we know that  
\begin{equation*}
    \omega^{2}_{P_{\ell},n_{0},1} \, \epsilon_{0}(z_{1}) \, \left(\epsilon_{0}(z_{1}) \, - \, \Lambda_{n_{0},j}\left( \omega_{P_{\ell},n_{0},1} \right) \right) \, \left(  \langle V\left(z_{1},  \theta, q \right), \int_{B}  e_{n_{0}}^{(3)}(y)  \, dy \rangle \right)^{2} \, \sim \, 1 . 
\end{equation*}
Clearly, from $(\ref{J1})$, we observe that the argmax\footnote{We recall that the argmax is the input point at which a function's output value is maximized.} of the function $\mathcal{J}\left(\cdot, z_{1} \right)$ is exactly the quasi-root\footnote{As the dispersion equation contains a small imaginary part, we refer to the root of its real part as a quasi-root.} of the first dispersion equation $\Lambda_{n_{0},1}\left( \cdot \right)$, given by $(\ref{Thm-Dis-Eq})$, and vice versa. Thus, we can estimate the first plasmonic resonance $\omega_{P_{\ell},n_{0},1}$. Moreover, by solving $(\ref{Thm-Dis-Eq})$, we deduce 
\begin{equation*}
    \epsilon_{0}\left( z_{1} \right) \, = \, \epsilon_{p}\left( \omega_{P_{\ell},n_{0},1} \right) \, \frac{\lambda_{n_{0}}^{(3)}\left(B\right)}{\left( \lambda_{n_{0}}^{(3)}\left(B\right) \, - \, 1 \right)}. 
\end{equation*}
This justify the reconstruction of $\epsilon_{0}\left( z_{1} \right)$. 
\item[]
\item Step 3)\label{StepIII}. Reconstructing $\mathcal{J}\left(\omega, z_{j} \right)$ and $\epsilon_{0}\left( z_{j} \right)$, for $2 \leq j \leq \aleph$. \\
\newline 
As explained through Step \ref{STEPI}, by injecting the first plasmonic nano-particle $D_{1}$ and measuring the contrast between the far-fields that we obtain before injecting $D_{1}$ and after injecting $D_{1}$, we can reconstruct the first imaging functional $\mathcal{J}\left(\cdot, z_{1} \right)$, see for instance $(\ref{Rec-J1})$. After that, we inject the second plasmonic nano-particle $D_{2}$ and we return to $(\ref{C-F-F-B-S-D})$, to obtain  
\begin{equation*}
\langle \left( E^{\infty} \, - \, V^{\infty}\right)(- \, \theta, \theta, q, \omega) , \left( \theta \times q \right) \rangle  \, \simeq \,-\, \frac{\mu \, a^{3}}{4 \, \pi \, \lambda_{n_{0}}^{(3)}(B) }  \, \sum_{j=1}^{2}  \, \mathcal{J}\left(\omega, z_{j} \right), 
\end{equation*}
or, equivalently, 
\begin{equation*}
\mathcal{J}\left(\omega, z_{2} \right)  \, \simeq \, - \, \mathcal{J}\left(\omega, z_{1} \right) \, - \, \frac{4 \, \pi \, \lambda_{n_{0}}^{(3)}(B)}{\mu \, a^{3}} \, \langle  \left( E^{\infty} \, - \, V^{\infty}\right)(- \, \theta, \theta, q, \omega) ; \left( \theta \times q \right) \rangle.
\end{equation*}
The right hand side of the formula above is already known (measured). Therefore, we deduce the reconstruction of the second imaging functional, i.e. 
\begin{equation}\label{Rec-J2}
    \omega \longrightarrow \mathcal{J}\left(\omega, z_{2} \right), \quad \text{with} \quad \omega \in \mathbf{I}.
\end{equation}
Additionally, by using the same arguments as those used in Step (\ref{STEPII}), we can determine the reconstruction of $\epsilon_{0}(z_{2})$ from $\mathcal{J}\left(\cdot, z_{2} \right)$, see for instance $\left( \ref{Rec-J2} \right)$. Hence, we have reconstructed 
\begin{equation*}
    \left(\mathcal{J}\left(\cdot, z_{2} \right); \epsilon_{0}(z_{2}) \right). 
\end{equation*}
Finally, through an induction process, as explained in Step (\ref{STEPI})-Step (\ref{StepIII}), we can reconstruct 
\begin{equation*}
    \left(\mathcal{J}\left(\cdot, z_{j} \right); \epsilon_{0}(z_{j}) \right)_{j=2}^{\aleph}. 
\end{equation*}
\item[]
\item Step 4). Reconstructing $\epsilon_{0}(\cdot)$ within $\Omega$. \\
\newline
To reconstruct the permittivity function $\epsilon_{0}(\cdot)$, inside $\Omega$, or at least approximate it, from its already reconstructed point-wise values $\left\{ \epsilon_{0}\left( z_{j} \right) \right\}_{j=1}^{\aleph}$, we can use the \textit{Dual Reciprocity Method}. Other interpolation methods can also be used. For a more detailed study of the \textit{Dual Reciprocity Method} and its application, \cite{Iran, cruse2012} and \cite{partridge2012dual} are recommended. This is a description to do so. \\ \newline 
We start by assuming that $\epsilon_{0}\left(\cdot\right)$ is smooth enough and we approximate it by a linear combination of a finite number of basis functions\footnote{The accuracy of the approximation $(\ref{Vexpres})$ depends on the choice of $\left\{ f_{k}(\cdot) \right\}_{k=1}^{\aleph}$ and the number $\aleph$.}, i.e.  
\begin{equation}\label{Vexpres}
    \epsilon_{0}\left( z \right) \, = \, \sum_{k=1}^{\aleph} \beta_{k} \, f_{k}(z),\quad z \in \Omega,
\end{equation}
with $\left(\beta_{1}, \cdots, \beta_{\aleph} \right) \in \mathbb{C}^{\aleph}$ being a vector to be determined,  and $\left(f_{1}(\cdot), \cdots, f_{\aleph}(\cdot)\right)$ are suitably chosen basis functions. In the literature, there are several types of basis functions that have been created to handle with the \textit{Dual Reciprocity Method}. Among them, without being exhaustive, we can list
\begin{enumerate}
    \item[]
    \item  A linear basis function given by $f_{k}(\cdot)\, = \, 1 \, + \, \left\vert \cdot - x_{k} \right\vert$;
    \item[] 
    \item  A Gaussian basis function given by $f_{k}(\cdot) \, = \, e^{\left\vert \cdot - x_{k} \right\vert^{2}}$;
    \item[] 
    \item  A thin plate spline basis function given by $f_{k}(\cdot) \, = \, \left\vert \cdot - x_{k} \right\vert^{2}\ln \left(\left\vert \cdot - x_{k} \right\vert \right)$,
    \item[] 
\end{enumerate}
where the set of points $\left\{ x_{k} \right\}_{k=1}^{\aleph}$, called collocation points, is generated randomly inside $\Omega$. Thus, we know the set of points $\left\{ x_{k} \right\}_{k=1}^{\aleph}$. \\
\newline 
Now, by choosing a type for the basis functions, generating randomly a set of collocations points $\left\{ x_{k} \right\}_{k=1}^{\aleph}$ inside $\Omega$ and evaluating both sides of $(\ref{Vexpres})$ at the a-priori assumed known location $\left\{ z_{j} \right\}_{j=1}^{\aleph}$, we obtain the following algebraic system, 
\begin{equation}\label{Step4AS}
    \begin{pmatrix}
        f_{1}(z_{1}) & \cdots & f_{\aleph}(z_{1}) \\
        \vdots & \ddots & \vdots \\
        f_{1}(z_{\aleph}) & \cdots & f_{\aleph}(z_{\aleph})
    \end{pmatrix} \cdot \begin{pmatrix}
        \beta_{1} \\
        \vdots \\
        \beta_{\aleph}
    \end{pmatrix} \, = \, \begin{pmatrix}
        \epsilon_{0}(z_{1}) \\
        \vdots \\
        \epsilon_{0}(z_{\aleph})
    \end{pmatrix}.
\end{equation}
Thanks to Step (\ref{STEPI}) - Step (\ref{StepIII}), the right hand side of $(\ref{Step4AS})$ is a known vector. Moreover, by its construction, the matrix appearing on the left hand side of $(\ref{Step4AS})$ can be computed explicitly. Then, by solving $(\ref{Step4AS})$, we can determine the unknown vector $\left(\beta_{1}, \cdots, \beta_{\aleph} \right)$. The final stage involves inserting the obtained vector into $(\ref{Vexpres})$, to come up with an approximation of the permittivity function $\epsilon_{0}(\cdot)$ inside $\Omega$. 
\end{enumerate} 
\medskip
We conclude this section by noting the following remark. 
\begin{remark}
    When the $\aleph$-plasmonic nano-particles are plugged in $\Omega$ simultaneously, and the frequency $\omega$ in $\mathbf{I}$ is varied, it can be observed that the function  
    \begin{equation*}
        \omega \longrightarrow \langle \left( E^{\infty} \, - \, V^{\infty}\right)(- \, \theta, \theta, q, \omega) , \left( \theta \times q \right) \rangle, 
    \end{equation*}
    accepts $\aleph$ peaks. Unfortunately, it is unclear to us how to match the $j^{th}$ peak to the $j^{th}$ plasmonic resonance, for $1 \leq j \leq \aleph$. This makes it impossible to reconstruct the permittivity function $\epsilon_{0}(\cdot)$ using our proposed imaging method while injecting all the nano-particles at once. This is the main reason why we propose, instead, to inject the nano-particles one after another.
\end{remark}
\section{Proof of Theorem \ref{ThmResult}}\label{SecPrfThm}
We divide this section into two subsections. In the first subsection, we will focus on deriving the dominant term related to the optical electric field and show its dependence on the eigen-system related to the Magnetization operator, the permittivity of the medium to be imaged, i.e $\epsilon_{0}(\cdot)$, and the permittivity of the used plasmonic nano-particles, i.e. $\epsilon_{p}(\omega)$. In the second subsection, we shall use the results from the first subsection to estimate the scattered fields and total fields connected to the electric field.
\subsection{Deriving the dominant term of the optical field.}
As shown in \cite[Section 3.2]{AG-MS-Maxwell}, the solution of $(\ref{eq:electromagnetic_scattering})$ can be written as the solution to the following Lippmann-Schwinger system of equations
\begin{equation}\label{SK0}
u_{1}(x) + \omega^{2}  \, \int_{D} G_{k}(x,y) \cdot u_{1}(y) \, (\bm{n}_{0}^{2}(y) - \bm{n}^{2}(y) ) \, dy = u_{0}(x), \quad x \in \mathbb{R}^{3}, 
\end{equation} 
where $u_0(\cdot)$ (respectively. $u_1(\cdot)$)  denotes the electromagnetic field before (respectively. after) injecting the nano-particles, $D \, = \, \overset{\aleph}{\underset{j=1}\cup} D_{j}$, with $\aleph \sim a^{-s}$, $s > 0$, and $G_{k}(\cdot,\cdot)$ is solution to $(\ref{Eq0401})$.
%the Green's kernel for Maxwell's equation for the inhomogeneous background, defined as solution of 
%\begin{equation*}
%\underset{y}{\nabla} \times \underset{y}{\nabla} \times \left( G_{k} \right)(x,y) - \omega^{2} \, \bm{n}_{0}^{2}(y) \, G_{k}(x,y) \, = \, \underset{x}{\delta}(y) \, \bm{I}, \;\, x, y \in \mathbb{R}^{3}, 
%\end{equation*}
%such that each column of $G_{k}(\cdot,\cdot)$ satisfies the outgoing radiation condition 
%\begin{equation*}
%\lim_{\left\vert x \right\vert \rightarrow +\infty} \;\; \left\vert x \right\vert \;\; \left( \underset{y}{\nabla} \times \left( G_{k} \right)(x,y) \times \frac{x}{\left\vert x \right\vert} - i \, k \, G_{k}(x,y) \right) = 0,
%\end{equation*}
%\qq{with $k$ being the wave number given in \eqref{DefEinc}}.
\medskip
\newline
To give sense of the integral appearing on the left hand side of $(\ref{SK0})$, i.e.
\begin{equation}\label{Int(x)}
Int(x) \, := \, \int_{D} G_{k}(\cdot,y) \cdot u_{1}(y) \, (\bm{n}_{0}^{2}(y) - \bm{n}^{2}(y) ) \, dy, \quad x \in \mathbb{R}^{3}, 
\end{equation}
we start by recalling that 
\begin{equation}\label{SG}
    G_{k}(x, \cdot) \; = \; \Pi_{k}(x, \cdot) \, + \, \Gamma(x, \cdot), 
\end{equation}
see \cite[Subsection 3.2]{AG-MS-Maxwell}, where 
\begin{equation}\label{DefPi=}
    \Pi_{k}(x,y) \, := \, \frac{1}{k^{2}} \, \underset{y}{\nabla} \, \underset{y}{\nabla} \, \Phi_{k}(x,y) \, + \, \Phi_{k}(x,y) \, I_{3},  
\end{equation}
and $\Gamma(x, \cdot) \, \in \, \mathbb{L}^{\frac{3}{2} \, - \, \delta}\left( \Omega \right)$, with $\delta$ being a small positive real number. Then, using $(\ref{SG}), (\ref{DefPi=})$ and the definition of both the Magnetization operator and the Newtonian operator in $(\ref{DefNDefMk})$, $(\ref{Int(x)})$ should be understood as
\begin{eqnarray}\label{Int(x)-Gamma}
\nonumber
Int(x) \, &=& \, - \frac{1}{k^{2}} \, \nabla M^{k}_{D}\left(  u_{1}(\cdot) \, (\bm{n}_{0}^{2}(\cdot) - \bm{n}^{2}(\cdot) )\right)(x) \, + \, N^{k}_{D}\left(  u_{1}(\cdot) \, (\bm{n}_{0}^{2}(\cdot) - \bm{n}^{2}(\cdot) )\right)(x) \\ &+& \,  \int_{D} \Gamma(x,y) \cdot u_{1}(y) \, (\bm{n}_{0}^{2}(y) - \bm{n}^{2}(y) ) \, dy, \quad x \in \mathbb{R}^{3}.
\end{eqnarray}
It is clear that the first term and the second term on the right hand side of $(\ref{Int(x)-Gamma})$ are well defined. In addition, the existence of the third term on the right hand side of $(\ref{Int(x)-Gamma})$ can be deduced by using the fact that $\Gamma(x, \cdot) \, \in \, \mathbb{L}^{\frac{3}{2} \, - \, \delta}\left( \Omega \right)$. This provides an explanation of how $(\ref{Int(x)})$ can be interpreted. Therefore, from now on whenever we use the notation $(\ref{Int(x)})$, it should be understood as $(\ref{Int(x)-Gamma})$.  
% The formal representation of the convolution part is justified in \cite[Section 3.2]{AG-MS-Maxwell} as well.
\bigskip
\newline
Recalling the definition of the index of refraction $\bm{n}(\cdot)$, see $(\ref{Def-Index-Ref})$, we rewrite $(\ref{SK0})$ as
\begin{equation}\label{SK}
u_{1}(x) + \omega^{2} \, \mu \, \int_{D} G_{k}(x,y) \cdot u_{1}(y) \, (\epsilon_{0}(y)-\epsilon_{p}(\omega)) \, dy = u_{0}(x), \quad x \in \mathbb{R}^{3}, 
\end{equation} 
where $\epsilon_{p}(\cdot)$ is the given permittivity of the injected nano-particles, through the Lorentz model, by $(\ref{Lorentz-model})$. Then, by restricting $x \in D_{m}$, we rewrite $(\ref{SK})$ as 
\begin{equation*}
u_{1}(x) \, + \, \omega^{2} \, \mu \, \int_{D_{m}} G_{k}(x,y) \cdot u_{1}(y) \, (\epsilon_{0}(y)-\epsilon_{p}(\omega)) \, dy \, + \, \omega^{2} \, \mu \, \sum_{j=1 \atop j \neq m}^{\aleph} \int_{D_{j}} G_{k}(x,y) \cdot u_{1}(y) \, (\epsilon_{0}(y)-\epsilon_{p}(\omega)) \, dy \, = \, u_{0}(x).
\end{equation*}
By expanding the function $(\epsilon_{0}(\cdot) \, - \, \epsilon_{p}(\omega))$ near the centers, we obtain  \begin{eqnarray}\label{SL}
\nonumber
u_{1}(x) &+& \omega^{2} \, \mu \, (\epsilon_{0}(z_{m})-\epsilon_{p}(\omega)) \, \int_{D_{m}} G_{k}(x,y) \cdot u_{1}(y) \,  dy \, + \, \omega^{2} \, \mu \, \sum_{j=1 \atop j \neq m}^{\aleph} \, (\epsilon_{0}(z_{j})-\epsilon_{p}(\omega)) \, \int_{D_{j}} G_{k}(x,y) \cdot u_{1}(y)  \, dy \\ \nonumber &=& u_{0}(x) - \omega^{2} \, \mu \, \int_{D_{m}} G_{k}(x,y) \cdot u_{1}(y)  \, \int_{0}^{1} \nabla \epsilon_{0}(z_{m}+t(y-z_{m})) \cdot (y - z_{m}) \, dt  \, dy \\
&-& \omega^{2} \, \mu \, \sum_{j=1 \atop j \neq m}^{\aleph}  \, \int_{D_{j}} G_{k}(x,y) \cdot u_{1}(y) \, \int_{0}^{1} \nabla \epsilon_{0}(z_{j}+t(y-z_{j})) \cdot (y - z_{j}) \, dt \, dy.
\end{eqnarray}
Thanks to \cite[Theorem 2.1, Formula (2.3)]{AG-MS-Maxwell}, regarding the Green's kernel $G_{k}(\cdot,\cdot)$, we have the following expansion\footnote{To comprehend $(\ref{Int(x)})$, we made the choice to decompose the Green's kernel $G_{k}(\cdot, \cdot)$ as indicated by $(\ref{SG})$ rather than $(\ref{SLI})$. This is motivated by our desire to write $(\ref{Int(x)})$ in a more explicit form, see $(\ref{Int(x)-Gamma})$. Nevertheless, the decomposition $(\ref{SLI})$ can be utilized to comprehend $(\ref{Int(x)})$.}
\begin{equation}\label{SLI}
    G_{k}(x,y) \, = \, \frac{1}{\omega^{2} \, \mu \, \epsilon_{0}(y)} \, \underset{x}{\nabla} \, \underset{x}{\nabla} \Phi_{0}(x,y) \, + \, \Gamma(x,y), \quad x \neq y, 
\end{equation}
where the term $\Gamma(\cdot, \cdot)$ is given by   
\begin{equation}\label{SLII}
    \Gamma(x,y) \, = \, \frac{- \, 1}{\omega^{2} \, \mu \, \left( \epsilon_{0}(y) \right)^{2}} \underset{x}{\nabla} \underset{x}{\nabla} M\left( \Phi_{0}(\cdot, y) \nabla \epsilon_{0}(y) \right)(x) \, + \, \Xi(x,z) ,  
\end{equation}
with $\nabla M(\cdot)$ being the Magnetization operator defined by $(\ref{DefNDefM})$ and $\Xi(\cdot,z) \in \mathbb{L}^{3-\delta}(D)$, where $\delta$ is a small positive parameter. Then, by denoting $$\tau_{m} := \frac{(\epsilon_{0}(z_{m}) - \epsilon_{p}(\omega))}{\epsilon_{0}(z_{m})},$$
using $(\ref{SLI})$ and $(\ref{SLII})$, we derive from $(\ref{SL})$ the following equation 
\begin{equation}\label{HA1}
    \left[ I \, - \, \tau_{m} \, \nabla M_{D_{m}} \right](u_{1})(x) \, + \, \omega^{2} \, \mu \, \sum_{j=1 \atop j \neq m}^{\aleph}  \tau_{j} \, \epsilon_{0}(z_{j}) \, \int_{D_{j}} G_{k}(x,y) \cdot u_{1}(y) \, dy \, = \, u_{0}(x) \, + \, Err_{0}(x),
\end{equation}
where
\begin{eqnarray}\label{FM1}
\nonumber
     Err_{0}(x) &:=& - \, \omega^{2} \, \mu \, \sum_{j=1 \atop j \neq m}^{\aleph} \int_{D_{j}} G_{k}(x,y) \cdot u_{1}(y) \int_{0}^{1} \nabla(\epsilon_{0})(z_{j}+t(y-z_{j})) \cdot (y - z_{j}) \, dt \, dy 
    \\ \nonumber &-& \, \omega^{2} \, \mu \,  \int_{D_{m}} G_{k}(x,y) \cdot u_{1}(y) \int_{0}^{1} \nabla(\epsilon_{0})(z_{m}+t(y-z_{m})) \cdot (y - z_{j}) \, dt \, dy \\ \nonumber
    &-& \omega^{2} \, \mu \, \left( \epsilon_{0}(z_{m}) - \epsilon_{p}(\omega) \right) \, \int_{D_{m}} \Gamma(x,y) \cdot u_{1}(y) \,  dy \\
    &+& (\epsilon_{0}(z_{m}) - \epsilon_{p}(\omega)) \, \nabla M_{D_{m}} \left(u_{1}(\cdot) \, \int_{0}^{1} \nabla(\epsilon_{0}^{-1})(z_{m}+t(\cdot - z_{m})) \cdot (\cdot - z_{m}) \, dt \right)(x).
\end{eqnarray}
Successively, by taking the inverse of the operator\footnote{The invertibility of the operator $\left[ I \, - \, \tau_{m} \, \nabla M_{D_{m}} \right]$ is a consequence of the fact that $\tau^{-1}_{m} \notin \sigma\left( \nabla M_{D_{m}} \right) \cup \left\{ 1 \right\}$, where $\sigma\left(\cdot \right)$ stands for the spectrum set.} $\left[ I \, - \, \tau_{m} \, \nabla M_{D_{m}} \right]$ on the both sides of the equation $(\ref{HA1})$, integrating over the domain $D_{m}$ and using the definition
\begin{equation*}
    W_{m}(\cdot) \, := \, \left[ I \, - \, \tau_{m} \, \nabla M_{D_{m}} \right]^{-1}(I)(\cdot), \quad \text{in} \; D_{m}, 
\end{equation*}
we obtain 
\begin{eqnarray}\label{HA2}
\nonumber
       \int_{D_{m}} u_{1}(x) \, dx &+& \omega^{2} \, \mu \, \sum_{j=1 \atop j \neq m}^{\aleph}  \tau_{j} \, \epsilon_{0}(z_{j}) \, \int_{D_{m}} W_{m}(x) \cdot \int_{D_{j}} G_{k}(x,y) \cdot u_{1}(y) \, dy \, dx \\ &=& \int_{D_{m}} W_{m}(x) \cdot u_{0}(x) \, dx \, + \, \int_{D_{m}} W_m(x) \cdot Err_{0}(x) \, dx.
\end{eqnarray}
Next, by setting 
\begin{equation*}
    \mathcal{C}_{m} \, := \, \int_{D_{m}} W_{m}(x) \, dx, 
\end{equation*}
and expanding both the Green's kernel $G_{k}(\cdot, \cdot)$ and the vector field $u_{0}(\cdot)$ near the centers, we derive from $(\ref{HA2})$ the following algebraic system 
\begin{equation}\label{as0354}
    \mathcal{C}_{m}^{-1} \cdot \int_{D_{m}} u_{1}(x) \, dx \, + \, \omega^{2} \, \mu \, \sum_{j = 1 \atop j \neq m}^{\aleph} \tau_{j} \, \epsilon_{0}(z_{j}) \, G_{k}(z_{m}, z_{j}) \cdot  \mathcal{C}_{j} \cdot \mathcal{C}_{j}^{-1} \cdot \int_{D_{j}} u_{1}(y) \, dy \, = \, u_{0}(z_{m}) \, + \, \mathcal{C}_{m}^{-1} \cdot Error_{1,m},
\end{equation}
where
\begin{eqnarray}\label{Errormexpression}
\nonumber
    Error_{1,m} &:=& - \, \omega^{2} \, \mu \, \sum_{j=1 \atop j \neq m}^{\aleph} \tau_{j} \, \epsilon_{0}(z_{j}) \,  \int_{D_{m}} W_{m}(x) \cdot \int_{D_{j}} \int_{0}^{1} \nabla G_{k}(x,z_{j}+t(y-z_{j})) \cdot (y - z_{j}) \, dt \cdot u_{1}(y) \, dy \, dx \\ \nonumber
    &-&  \omega^{2} \, \mu \, \sum_{j=1 \atop j \neq m}^{\aleph} \tau_{j} \, \epsilon_{0}(z_{j}) \,  \int_{D_{m}} W_{m}(x) \cdot  \int_{0}^{1} \nabla G_{k}(z_{m}+t(x-z_{m}),z_{j}) \cdot (x - z_{m}) \, dt \, dx \cdot \int_{D_{j}} u_{1}(y) \, dy  \\ 
    &+&  \int_{D_{m}} W_{m}(x) \cdot  \int_{0}^{1} \nabla u_{0}(z_{m}+t(x-z_{m})) \cdot (x - z_{m}) \, dt \, dx \, + \, \int_{D_{m}} W_{m}(x) \cdot Err_{0}(x) \, dx.
\end{eqnarray}
Moreover, for $(\ref{as0354})$, we correspond to the following matrix form,
\begin{equation*}
    \begin{pmatrix}
       \mathcal{Q}_{1} \\
        \mathcal{Q}_{2} \\
        \vdots \\
        \mathcal{Q}_{\aleph}
    \end{pmatrix} \, + \, \omega^{2} \, \mu \, \mathcal{M}  \cdot   \begin{pmatrix}
       \mathcal{Q}_{1} \\
        \mathcal{Q}_{2} \\
        \vdots \\
        \mathcal{Q}_{\aleph}
    \end{pmatrix} \, = \, \begin{pmatrix}
       u_{0}(z_{1}) \\
       u_{0}(z_{2}) \\
        \vdots \\
       u_{0}(z_{\aleph})
    \end{pmatrix} \, + \, \begin{pmatrix}
        \mathcal{C}_{1}^{-1} \cdot Error_{1,1} \\
        \mathcal{C}_{2}^{-1} \cdot Error_{1,2} \\
        \vdots \\
        \mathcal{C}_{\aleph}^{-1} \cdot Error_{1,\aleph}
    \end{pmatrix},
\end{equation*}
where $\mathcal{M}$ is the matrix given by 
\begin{equation*}
    \mathcal{M} := \begin{pmatrix}
        0 & G_{k}(z_{1},z_{2}) \cdot \tau_{2} \epsilon_{0}(z_{2}) \, \mathcal{C}_{2} & \cdots & G_{k}(z_{1},z_{\aleph}) \cdot \tau_{\aleph} \epsilon_{0}(z_{\aleph}) \, \mathcal{C}_{\aleph} \\
        G_{k}(z_{2},z_{1}) \cdot \tau_{1} \, \epsilon_{0}(z_{1}) \mathcal{C}_{1} & 0 & \cdots & G_{k}(z_{2},z_{\aleph}) \cdot \tau_{\aleph} \epsilon_{0}(z_{\aleph}) \, \mathcal{C}_{\aleph} \\
        \vdots & \vdots & \ddots & \vdots \\
        G_{k}(z_{\aleph},z_{1}) \cdot \tau_{1} \epsilon_{0}(z_{1})\, \mathcal{C}_{1} & G_{k}(z_{\aleph},z_{2}) \cdot \tau_{2} \epsilon_{0}(z_{2})\, \mathcal{C}_{2} & \cdots & 0 \\
    \end{pmatrix},
\end{equation*}
$Error_{1,m}$ is given by $(\ref{Errormexpression})$, and 
\begin{equation}\label{QCU}
    \mathcal{Q}_{m} \, := \, \mathcal{C}_{m}^{-1} \cdot \int_{D_{m}} u_{1}(x) \, dx, \quad 1 \leq m \leq \aleph. 
\end{equation}
To prove the invertibility of the above algebraic system, it is sufficient to derive a condition such that the matrix appearing on the L.H.S is a diagonal dominant block matrix, i.e. 
\begin{equation}\label{Cdt0508}
    \omega^{2} \, \mu \, \sum_{j=1 \atop j \neq m}^{\aleph} \left\vert G_{k}(z_{m},z_{j}) \cdot \tau_{j} \epsilon_{0}(z_{j}) \, \mathcal{C}_{j} \right\vert \, < \, \left\vert I_{3} \right\vert.
\end{equation}
To do this, we have 
\begin{equation*}
    \omega^{2} \, \mu \, \sum_{j=1 \atop j \neq m}^{\aleph} \left\vert G_{k}(z_{m},z_{j}) \cdot \tau_{j} \epsilon_{0}(z_{j})\, \mathcal{C}_{j} \right\vert \, \lesssim \, \sum_{j=1 \atop j \neq m}^{\aleph} \left\vert G_{k}(z_{m},z_{j}) \right\vert \, \left\vert \tau_{j} \right\vert \, \left\vert \epsilon_{0}(z_{j}) \right\vert \, \left\vert \mathcal{C}_{j} \right\vert \, \lesssim \, \sum_{j=1 \atop j \neq m}^{\aleph} \left\vert G_{k}(z_{m},z_{j}) \right\vert  \, \left\vert \mathcal{C}_{j} \right\vert.
\end{equation*}
In addition, thanks to \cite[Proposition 2.2]{AG-MS-Maxwell}, we know that
\begin{equation}\label{BTD}
\mathcal{C}_{j} \, = \, \mathcal{O}\left( a^{3-h} \right),    
\end{equation}
and this implies 
\begin{equation*}
    \omega^{2} \, \mu \, \sum_{j=1 \atop j \neq m}^{\aleph} \left\vert G_{k}(z_{m},z_{j}) \cdot \tau_{j} \epsilon_{0}(z_{j}) \, \mathcal{C}_{j} \right\vert  \, \lesssim \, a^{3-h} \, \sum_{j=1 \atop j \neq m}^{\aleph} \frac{1}{d^{3}_{mj}} \, = \, \mathcal{O}\left( a^{3-h} \, d^{-3} \, \aleph \right) \, = \, \mathcal{O}\left( a^{3 \, - \, h \, - \, 3 t \, - \, s} \right). 
\end{equation*}
Hence, under the condition 
\begin{equation}\label{addittop15}
    3 \, - \, h \, - \, 3 t \, - \, s \, > \, 0, 
\end{equation} 
the condition $(\ref{Cdt0508})$ will be satisfied. By using Born series and keeping the dominant terms, we get 
\begin{equation}\label{LZI}
    \begin{pmatrix}
       \mathcal{Q}_{1} \\
        \mathcal{Q}_{2} \\
        \vdots \\
        \mathcal{Q}_{\aleph}
    \end{pmatrix} \, 
    = \, \begin{pmatrix}
       u_{0}(z_{1}) \\
       u_{0}(z_{2}) \\
        \vdots \\
       u_{0}(z_{\aleph})
    \end{pmatrix} \, + \, \begin{pmatrix}
         Error_{2,1} \\
         Error_{2,2} \\
        \vdots \\
         Error_{2,\aleph}
    \end{pmatrix},  
    \end{equation}
    where 
    \begin{equation*}
    \begin{pmatrix}
         Error_{2,1} \\
         Error_{2,2} \\
        \vdots \\
         Error_{2,\aleph}
    \end{pmatrix} \, := \, \begin{pmatrix}
        \mathcal{C}_{1}^{-1} \cdot Error_{1,1} \\
        \mathcal{C}_{2}^{-1} \cdot Error_{1,2} \\
        \vdots \\
        \mathcal{C}_{\aleph}^{-1} \cdot Error_{1,\aleph}
    \end{pmatrix} \, - \, \omega^{2} \, \mu \, \mathcal{M}  \cdot   \begin{pmatrix}
       \mathcal{Q}_{1} \\
        \mathcal{Q}_{2} \\
        \vdots \\
        \mathcal{Q}_{\aleph}
    \end{pmatrix},
\end{equation*}
which can be rewritten as 
\begin{equation*}
  \begin{pmatrix}
      Error_{2,1} \\
      Error_{2,2} \\
      \vdots \\
      Error_{2,\aleph}
  \end{pmatrix}   
    =
    \begin{pmatrix}
        \mathcal{C}_{1}^{-1} \cdot Error_{1,1} \, - \, \omega^{2} \, \mu \, \sum_{j=1 \atop j \neq 1}^{\aleph}G_{k}(z_{1},z_{j}) \cdot \tau_{j} \epsilon_{0}(z_{j}) \, \mathcal{C}_{j} \, \mathcal{Q}_{j} \\
        \mathcal{C}_{2}^{-1} \cdot Error_{1,2} \, - \, \omega^{2} \, \mu \, \sum_{j=1 \atop j \neq 2}^{\aleph}G_{k}(z_{2},z_{j}) \cdot \tau_{j} \epsilon_{0}(z_{j}) \, \mathcal{C}_{j} \, \mathcal{Q}_{j} \\
        \vdots \\
        \mathcal{C}_{\aleph}^{-1} \cdot Error_{1,\aleph} \, - \, \omega^{2} \, \mu \, \sum_{j=1 \atop j \neq \aleph}^{\aleph}G_{k}(z_{\aleph},z_{j}) \cdot \tau_{j} \epsilon_{0}(z_{j}) \, \mathcal{C}_{j} \, \mathcal{Q}_{j} 
    \end{pmatrix}.
    \end{equation*}
    Next, we need to estimate the term $Error_{2,k}$, for $k \, = \, 1, \cdots, \aleph$. To do this, we recall that 
    \begin{eqnarray*}
    Error_{2,k} \; &:=& \; \mathcal{C}_{1}^{-1} \cdot Error_{1,k} \, - \, \omega^{2} \, \mu \, \sum_{j=1 \atop j \neq k}^{\aleph}G_{k}(z_{k},z_{j}) \cdot \tau_{j} \epsilon_{0}(z_{j}) \, \mathcal{C}_{j} \, \mathcal{Q}_{j} \\
    Error_{2,k} \; &\overset{(\ref{QCU})}{=}& \; \mathcal{C}_{1}^{-1} \cdot Error_{1,k} \, - \, \omega^{2} \, \mu \, \sum_{j=1 \atop j \neq k}^{\aleph} \tau_{j} \, \epsilon_{0}(z_{j}) \, G_{k}(z_{k},z_{j}) \cdot \int_{D_{j}} u_{1}(x) \, dx \\
    \left\vert Error_{2,k} \right\vert \; & \lesssim & \; \left\vert \mathcal{C}_{1}^{-1} \right\vert \, \left\vert Error_{1,k} \right\vert \, + \, \sum_{j=1 \atop j \neq k}^{\aleph} \, \left\vert G_{k}(z_{k},z_{j}) \right\vert \, \left\Vert 1 \right\Vert_{\mathbb{L}^{2}(D_{j})} \, \left\Vert u_{1} \right\Vert_{\mathbb{L}^{2}(D_{j})} \\
    & \overset{(\ref{BTD})}{\lesssim} & \; a^{h-3} \, \left\vert Error_{1,k} \right\vert \, + \, a^{\frac{3}{2}} \, \sum_{j=1 \atop j \neq k}^{\aleph} \, \frac{1}{d^{3}_{kj}}  \, \left\Vert u_{1} \right\Vert_{\mathbb{L}^{2}(D_{j})} \\
    & \lesssim & \; a^{h-3} \, \left\vert Error_{1,k} \right\vert \, + \, a^{\frac{3}{2}} \, \left\Vert u_{1} \right\Vert_{\mathbb{L}^{2}(D)} \, \left( \sum_{j=1 \atop j \neq k}^{\aleph} \, \frac{1}{d^{6}_{kj}} \right)^{\frac{1}{2}}   \\
    & \lesssim & \; a^{h-3} \, \left\vert Error_{1,k} \right\vert \, + \, a^{\frac{3}{2}} \, \left\Vert u_{1} \right\Vert_{\mathbb{L}^{2}(D)} \, d^{-3} \, \aleph^{\frac{1}{2}}.
    \end{eqnarray*}
    Knowing that 
\begin{equation}\label{ape}
\left\Vert u_{1} \right\Vert_{\mathbb{L}^{2}(D)} \, = \, \mathcal{O}\left( a^{\frac{3}{2}-h-\frac{s}{2}} \right),
\end{equation}
see Proposition \ref{Propapest} for its justification, we deduce that 
\begin{equation}\label{Eq0638}
    \left\vert Error_{2,k} \right\vert \; \lesssim \; a^{h-3} \, \left\vert Error_{1,k} \right\vert \, + \, a^{3-h-3t-{s}}.
\end{equation}
To accomplish the estimation of $(\ref{Eq0638})$, we need to estimate $\left\vert Error_{1,k} \right\vert$. For $(\ref{Errormexpression})$, we take the modulus on its both sides to get 
\begin{eqnarray*}
   \left\vert Error_{1,m} \right\vert & \lesssim &  \left\Vert W_{m} \right\Vert_{\mathbb{L}^{2}(D_{m})}  \Bigg[  \left\Vert 1 \right\Vert_{\mathbb{L}^{2}(D_{m})}  \sum_{j=1 \atop j \neq m}^{\aleph} \left\Vert u_{1} \right\Vert_{\mathbb{L}^{2}(D_{j})} \frac{1}{d^{4}_{mj}}  \left[  \int_{D_{j}} \left\vert y - z_{j} \right\vert^{2}  dy \right]^{\frac{1}{2}} \\ \nonumber
    &+&   \left\Vert 1 \right\Vert_{\mathbb{L}^{2}(D_{m})}\, \left[  \int_{D_{m}} \left\vert x - z_{m}  \right\vert^{2}  dx \right]^{\frac{1}{2}}  \sum_{j=1 \atop j \neq m}^{\aleph} \frac{1}{d^{4}_{mj}}   \left\Vert u_{1} \right\Vert_{\mathbb{L}^{2}(D_{j})}  +  \left[  \int_{D_{m}} \left\vert x - z_{m} \right\vert^{2} \, dx \right]^{\frac{1}{2}}  +  \left\Vert Err_{0} \right\Vert_{\mathbb{L}^{2}(D_{m})}\Bigg],
    \end{eqnarray*}
    where we have used the  fact that $u_{0}(\cdot)\in \mathcal{C}^1$ and the singularity of the Green's kernel $G_{k}(\cdot, \cdot)$ is of order 3. Besides, by the use of 
    \begin{equation*}
        \int_{D_{j}} \left\vert x - z_{j} \right\vert^{2} \, dx \, = \, \mathcal{O}\left( a^{5} \right), \quad 1 \leq j \leq \aleph, 
    \end{equation*}
    we obtain  
    \begin{eqnarray}\label{Errormexpressionms}
    \nonumber
   \left\vert Error_{1,m} \right\vert \, & \lesssim & \, \left\Vert W_{m} \right\Vert_{\mathbb{L}^{2}(D_{m})}  \, \left[ a^{4} \,  \sum_{j=1 \atop j \neq m}^{\aleph} \left\Vert u_{1} \right\Vert_{\mathbb{L}^{2}(D_{j})} \frac{1}{d^{4}_{mj}} \, + a^{\frac{5}{2}} \, + \, \left\Vert Err_{0} \right\Vert_{\mathbb{L}^{2}(D_{m})} \right] \\ \nonumber
   & \lesssim & \, \left\Vert W_{m} \right\Vert_{\mathbb{L}^{2}(D_{m})}  \, \left[ a^{4} \, d^{-4} \,  \left\Vert u_{1} \right\Vert_{\mathbb{L}^{2}(D)} \, \aleph^{\frac{1}{2}} \, + \, a^{\frac{5}{2}} \, + \, \left\Vert Err_{0} \right\Vert_{\mathbb{L}^{2}(D_{m})} \right] \\
   & \overset{(\ref{ape})}{\lesssim} & \, \left\Vert W_{m} \right\Vert_{\mathbb{L}^{2}(D_{m})}  \, \left[ a^{\frac{11}{2}-h-4t-s} \,  + \, a^{\frac{5}{2}} \, + \, \left\Vert Err_{0} \right\Vert_{\mathbb{L}^{2}(D_{m})} \right].
    \end{eqnarray}
    To estimate of  $\left\Vert Err_{0} \right\Vert_{\mathbb{L}^{2}(D_{m})}$, we go back to $(\ref{FM1})$ and take the $\left\Vert \cdot \right\Vert_{\mathbb{L}^{2}(D_{m})}$-norm on the both sides of the equation to obtain 
    \begin{eqnarray*}
     \left\Vert Err_{0} \right\Vert_{\mathbb{L}^{2}(D_{m})} \, & \lesssim &  \sum_{j=1 \atop j \neq m}^{\aleph} \left\Vert \int_{D_{j}} G_{k}(\cdot,y) \cdot u_{1}(y) \int_{0}^{1} \nabla(\epsilon_{0})(z_{j}+t(y-z_{j})) \cdot (y - z_{j}) \, dt \, dy \right\Vert_{\mathbb{L}^{2}(D_{m})} 
    \\  &+&  \left\Vert \int_{D_{m}} G_{k}(\cdot,y) \cdot u_{1}(y) \int_{0}^{1} \nabla(\epsilon_{0})(z_{m}+t(y-z_{m})) \cdot (y - z_{j}) \, dt \, dy \right\Vert_{\mathbb{L}^{2}(D_{m})} \\ \nonumber
    &+& \left\Vert \int_{D_{m}} \Gamma(\cdot, y) \cdot u_{1}(y) dy \right\Vert_{\mathbb{L}^{2}(D_{m})} 
    + \left\Vert  \nabla M_{D_{m}} \left(u_{1} \int_{0}^{1} \nabla(\epsilon_{0}^{-1})(z_{m}+t(\cdot - z_{m})) \cdot (\cdot - z_{m})  dt \right)\right\Vert_{\mathbb{L}^{2}(D_{m})},
\end{eqnarray*}
and, by using the smoothness of the function $\epsilon_{0}(\cdot)$, the dominant term related to the Green's kernel, see $(\ref{SLI})$, and the dominant term related to the kernel $\Gamma(\cdot, \cdot)$, see $(\ref{SLII})$, we reduce the previous estimation to 
    \begin{eqnarray*}
     \left\Vert Err_{0} \right\Vert_{\mathbb{L}^{2}(D_{m})} \, & \lesssim &  a \, \sum_{j=1 \atop j \neq m}^{\aleph} \left\Vert u_{1} \right\Vert_{\mathbb{L}^{2}(D_{j})} \,  \left[ \int_{D_{m}} \, \int_{D_{j}} \left\vert G_{k}(x,y) \right\vert^{2} \, dy \, dx  \right]^{\frac{1}{2}}
    \\  &+&  \left\Vert \nabla M_{D_{m}}\left( \frac{u_{1}(\cdot)}{\epsilon_{0}(\cdot)} \int_{0}^{1} \nabla(\epsilon_{0})(z_{m}+t(\cdot -z_{m})) \cdot (\cdot - z_{j}) \, dt  \right) \right\Vert_{\mathbb{L}^{2}(D_{m})} \\ \nonumber
    &+& \left\Vert \int_{D_{m}} \nabla \nabla M \left( \Phi_{0}(\cdot,y) \nabla \epsilon_{0}^{-1}(y) \right)(\cdot) \cdot u_{1}(y) \,  dy \right\Vert_{\mathbb{L}^{2}(D_{m})} \\
    &+& \left\Vert  \nabla M_{D_{m}} \left(u_{1}(\cdot) \, \int_{0}^{1} \nabla(\epsilon_{0}^{-1})(z_{m}+t(\cdot - z_{m})) \cdot (\cdot - z_{m}) \, dt \right)\right\Vert_{\mathbb{L}^{2}(D_{m})}.
\end{eqnarray*}
Moreover, since $\left\Vert \nabla M_{D_{m}} \right\Vert_{\mathcal{L}\left(\mathbb{L}^{2}(D_{m});\mathbb{L}^{2}(D_{m})\right)} \, = \, 1$, see \cite[Lemma 5.5]{AG-MS-Maxwell}, and using the fact that from singularity point of view we have   
\begin{equation*}
    \nabla \nabla M \left( \Phi_{0}(\cdot,y) \nabla \epsilon_{0}^{-1}(y) \right)(\cdot) \; \sim \; \nabla \left( \Phi_{0}(\cdot, y) \, \nabla \epsilon_{0}^{-1}(y) \right), 
\end{equation*}
we deduce the following estimation
    \begin{eqnarray*}
     \left\Vert Err_{0} \right\Vert_{\mathbb{L}^{2}(D_{m})} \, & \lesssim &  a \, \sum_{j=1 \atop j \neq m}^{\aleph} \left\Vert u_{1} \right\Vert_{\mathbb{L}^{2}(D_{j})} \,  \left[ \int_{D_{m}} \, \int_{D_{j}} \left\vert G_{k}(x,y) \right\vert^{2} \, dy \, dx  \right]^{\frac{1}{2}}
    \, + \, a \, \left\Vert  u_{1} \right\Vert_{\mathbb{L}^{2}(D_{m})} \\ \nonumber
    &+& \left\Vert \int_{D_{m}} \nabla  \left( \Phi_{0}(\cdot,y) \nabla \epsilon_{0}^{-1}(y) \right) \cdot u_{1}(y) \,  dy \right\Vert_{\mathbb{L}^{2}(D_{m})} \, + a \, \left\Vert  u_{1} \right\Vert_{\mathbb{L}^{2}(D_{m})}.
\end{eqnarray*}
On the R.H.S, to evaluate the first term we expand the Green's kernel $G_{k}(\cdot,\cdot)$ near the centers, together with the explicit computation for the third term, we end up with the following estimation
    \begin{equation*}
     \left\Vert Err_{0} \right\Vert_{\mathbb{L}^{2}(D_{m})} \,  \lesssim \,  a^{4} \, \sum_{j=1 \atop j \neq m}^{\aleph} \left\Vert u_{1} \right\Vert_{\mathbb{L}^{2}(D_{j})} \,   \left\vert G_{k}(z_{m},z_{j}) \right\vert
    \, + \, a \, \left\Vert  u_{1} \right\Vert_{\mathbb{L}^{2}(D_{m})} + \left\Vert \nabla N_{D_{m}}\left( \nabla \epsilon_{0}^{-1} \cdot u_{1} \right)  \right\Vert_{\mathbb{L}^{2}(D_{m})}.
\end{equation*}
Now, by using the fact that 
\begin{equation*}
    \left\Vert \nabla N_{D_{m}} \right\Vert_{\mathcal{L}\left(\mathbb{L}^{2}(D_{m}), \mathbb{L}^{2}(D_{m}) \right)} \; = \; \mathcal{O}\left( a \right) \quad \text{and} \quad     \left\vert G_{k}(z_{m},z_{j}) \right\vert \; \sim \; \frac{1}{d^{3}_{mj}},
\end{equation*}
we obtain  
    \begin{eqnarray}\label{FM219}
    \nonumber
     \left\Vert Err_{0} \right\Vert_{\mathbb{L}^{2}(D_{m})} \, & \lesssim & \,  a^{4} \, \sum_{j=1 \atop j \neq m}^{\aleph} \left\Vert u_{1} \right\Vert_{\mathbb{L}^{2}(D_{j})} \,  \frac{1}{d^{3}_{mj}}
    \, + \, a \, \left\Vert  u_{1} \right\Vert_{\mathbb{L}^{2}(D_{m})} \\ \nonumber
    & \lesssim & \,  a^{4} \, \left( \sum_{j=1 \atop j \neq m}^{\aleph} \left\Vert u_{1} \right\Vert^{2}_{\mathbb{L}^{2}(D_{j})} \right)^{\frac{1}{2}} \, \left( \sum_{j=1 \atop j \neq m}^{\aleph} \frac{1}{d^{6}_{mj}} \right)^{\frac{1}{2}}
    \, + \, a \, \left\Vert  u_{1} \right\Vert_{\mathbb{L}^{2}(D_{m})} \\ \nonumber
    & \lesssim & \, \left( a^{4} \, d^{-3} \, \aleph^{\frac{1}{2}} + \, a  \right) \,  \left\Vert  u_{1} \right\Vert_{\mathbb{L}^{2}(D)} \\ & \overset{(\ref{ape})}{=} & \, \mathcal{O}\left( a^{\frac{5}{2}-h-\frac{s}{2}} \right) \, + \, \mathcal{O}\left( a^{\frac{11}{2}-3t-s-h}\right) \overset{(\ref{addittop15})}{=}  \mathcal{O}\left( a^{\frac{5}{2}-h-\frac{s}{2}} \right).
\end{eqnarray}
By returning to $(\ref{Errormexpressionms})$ and using $(\ref{FM219})$, we deduce that 
 \begin{equation*}
   \left\vert Error_{1,m} \right\vert 
   \, \lesssim  \, \left\Vert W_{m} \right\Vert_{\mathbb{L}^{2}(D_{m})}  \, a^{\frac{5}{2}-h + \min\left(3 - 4 t - s; -\frac{s}2 \right)}.
    \end{equation*}
In addition, thanks to \cite[Proposition 2.2]{AG-MS-Maxwell}, we know that $\left\Vert W_{m} \right\Vert_{\mathbb{L}^{2}(D_{m})} \, = \, \mathcal{O}\left( a^{\frac{3}{2}-h} \right)$, and this implies 
\begin{equation*}
   \left\vert Error_{1,m} \right\vert  \, = \, \mathcal{O}\left( a^{(4 - 2 h) + \min\left(3 - 4 t - s; -\frac{s}2 \right)} \right).
\end{equation*}
Finally, by plugging the previous estimation into $(\ref{Eq0638})$, we obtain  
\begin{equation}\label{216217}
        \left\vert Error_{2,m} \right\vert  \, = \, \mathcal{O}\left( a^{\min((3 - 3t -  h - s);(1 -  h-\frac{s}{2}))} \right).
\end{equation}
Hence, by gathering $(\ref{LZI})$ and $(\ref{216217})$, we conclude that 
\begin{equation*}
    \mathcal{Q}_{m} \, = \, u_{0}(z_{m}) \, + \, \mathcal{O}\left( a^{\min((3 - 3t -  h - s);(1 -  h - \frac{s}{2}))} \right).
\end{equation*}
Consequently, by using $(\ref{QCU})$ and $(\ref{BTD})$, we obtain  
\begin{equation*}
    \int_{D_{m}} u_{1}(x) \, dx \, = \, \mathcal{C}_{m} \cdot u_{0}(z_{m}) \, + \, \mathcal{O}\left( a^{\min((6 - 3t - 2 h - s);(4 - 2 h -\frac{s}{2}))} \right).
\end{equation*}
Furthermore, from \cite[Proposition 2.2]{AG-MS-Maxwell}, we know that 
\begin{equation*}
    \mathcal{C}_{m} \, = \, \frac{a^{3} \, \epsilon_{0}(z_{m})}{\left(\epsilon_{0}(z_{m}) \, - \, \lambda_{n_{0}}^{(3)}(B) \, \left(\epsilon_{0}(z_{m}) \, - \, \epsilon_{p}(\omega) \right) \right)} \, \int_{B} e_{n_{0}}^{(3)}(x) \, dx \otimes \int_{B} e_{n_{0}}^{(3)}(x) \, dx \, + \, \mathcal{O}\left( a^{3} \right).
\end{equation*}
This implies, 
\begin{equation}\label{BRTVMM}
 \int_{D_{m}} u_{1}(x) dx \, = \, \frac{a^{3} \; \epsilon_{0}(z_{m}) \, \langle u_{0}(z_{m}) , \int_{B} e^{(3)}_{n_{0}}(x) \, dx \rangle}{\left(\epsilon_{0}(z_{m}) -  \left(\epsilon_{0}(z_{m}) - \epsilon_{p}(\omega) \right) \,   \lambda^{(3)}_{n_{0}}(B) \right)} \, \int_{B} e^{(3)}_{n_{0}}(x) \, dx + \, \mathcal{O}\left( a^{\min(3;(6 - 3t - 2 h - s);(4 - 2 h - \frac{s}{2}))} \right).
\end{equation}
\bigskip
\newline 
In the subsequent discussions, we will use the notation $V \, := \,u_{0} $, respectively $E\, := \, u_{1} $ for the solutions to $(\ref{eq:electromagnetic_scattering})$ signifying the electric field respectively before and after injecting the nano-particles $D$ inside $\Omega$. Taking into account the introduced notation, we rewrite $(\ref{BRTVMM})$ as  
\begin{equation}\label{BRTVMM+}
 \int_{D_{m}} E(x) dx \, = \, \frac{a^{3} \; \epsilon_{0}(z_{m}) \, \langle V\left( z_{m},  \theta, q \right) , \int_{B} e^{(3)}_{n_{0}}(x) \, dx \rangle}{\left(\epsilon_{0}(z_{m}) -  \left(\epsilon_{0}(z_{m}) - \epsilon_{p}(\omega) \right) \,   \lambda^{(3)}_{n_{0}}(B) \right)} \, \int_{B} e^{(3)}_{n_{0}}(x) \, dx + \, \mathcal{O}\left( a^{\min(3;(6 - 3t - 2 h - s);(4 - 2 h - \frac{s}{2}))} \right).
\end{equation}
\subsection{Estimation of the scattered  fields.}\label{ProofThemSubSecII}
To derive the estimate of the scattered field, we take $x$ away from $D$ in $(\ref{SK})$, 
\begin{equation*}
E(x) + \omega^{2} \, \mu \, \int_{D} G_{k}(x,y) \cdot E(y) \, \left( \epsilon_{0}(y) - \epsilon_{p}(\omega) \right) \, dy \, = \, V(x).
\end{equation*}
In addition, as $E \, = \, E^{I} \, + \, E^{s}$ and $V \, = \, V^{I} \, + \, V^{s}$, with $V^{I} \, = \, E^{I}$, we obtain 
\begin{eqnarray}\label{LFI}
\nonumber
E^{s}(x) - V^{s}(x) \, &=& \, - \, \omega^{2} \, \mu \, \int_{D} G_{k}(x,y) \cdot E(y) \, \left( \epsilon_{0}(y) - \epsilon_{p}(\omega) \right) \, dy \\
&=& \, - \, \omega^{2} \, \mu \, \sum_{j=1}^{\aleph} \int_{D_{j}} G_{k}(x,y) \cdot E(y) \, \left( \epsilon_{0}(y) - \epsilon_{p}(\omega) \right) \, dy,
\end{eqnarray}
which, by Taylor expansion for the permittivity function $\epsilon_{0}(\cdot)$, gives us 
\begin{equation*}
    E^{s}(x) \, - \, V^{s}(x) \, =  \, - \, \omega^{2} \, \mu \, \sum_{j=1}^{\aleph} \left( \epsilon_{0}(z_{j}) - \epsilon_{p}(\omega) \right) \, \int_{D_{j}} G_{k}(x,y) \cdot E(y)  \, dy \; - \; T_{1}(x),
\end{equation*}
where 
\begin{equation*}
T_{1}(x) := \, \omega^{2} \, \mu  \, \sum_{j=1}^{\aleph} \int_{D_{j}} G_{k}(x,y) \cdot E(y) \, \int_{0}^{1} \, \nabla \epsilon_{0}(z_{j}+t(y-z_{j})) \cdot (y-z_{j}) \, dt \, dy.
\end{equation*}
We need to estimate the term $T_{1}(\cdot)$. We have,  
\begin{eqnarray*}
    \left\vert T_{1}(x) \right\vert \, & \lesssim & \,   \sum_{j=1}^{\aleph} \, \int_{D_{j}} \left\vert  G_{k}(x,y) \cdot E(y) \, \int_{0}^{1} \, \nabla \epsilon_{0}(z_{j}+t(y-z_{j})) \cdot (y-z_{j}) \, dt \, \right\vert \,  dy \\
    & \lesssim & \, \sum_{j=1}^{\aleph} \,  \left\Vert  G_{k}(x,\cdot) \right\Vert_{\mathbb{L}^{2}(D_{j})} \, \left\Vert E \right\Vert_{\mathbb{L}^{2}(D_{j})} \, \left\Vert \int_{0}^{1} \, \nabla \epsilon_{0}(z_{j}+t(\cdot-z_{j})) \cdot (\cdot-z_{j}) \, dt \, \right\Vert_{\mathbb{L}^{\infty}(D_{j})}  \\
    & \lesssim & \, a \, \sum_{j=1}^{\aleph} \,  \left\Vert  G_{k}(x,\cdot) \right\Vert_{\mathbb{L}^{2}(D_{j})} \, \left\Vert E \right\Vert_{\mathbb{L}^{2}(D_{j})}  \, \lesssim  \, a \, \left( \sum_{j=1}^{\aleph} \,  \left\Vert  G_{k}(x,\cdot) \right\Vert^{2}_{\mathbb{L}^{2}(D_{j})} \right)^{\frac{1}{2}} \, \left\Vert E \right\Vert_{\mathbb{L}^{2}(D)},
\end{eqnarray*}
which, by knowing that $x$ is away from $D$, such that the Green's kernel $G_{k}(x,\cdot)$ is smooth, allows us to get 
\begin{equation*}
      \left\vert T_{1}(x) \right\vert \, = \, \mathcal{O}\left( a^{\frac{5-s}{2}} \; \left\Vert E \right\Vert_{\mathbb{L}^{2}(D)}\right). 
\end{equation*}
Then, 
\begin{equation*}
E^{s}(x)  \, - \, V^{s}(x) \, = \, - \, \omega^{2} \, \mu \, \sum_{j=1}^{\aleph} \left( \epsilon_{0}(z_{j}) - \epsilon_{p}(\omega) \right) \, \int_{D_{j}} G_{k}(x,y) \cdot E(y)  \, dy \, + \, \mathcal{O}\left( a^{\frac{5-s}{2}} \; \left\Vert E \right\Vert_{\mathbb{L}^{2}(D)}\right).
\end{equation*}
In addition, by Taylor expansion for $G_{k}(x,\cdot)$, we obtain
\begin{eqnarray*}
E^{s}(x) \, - \, V^{s}(x) \, &=& \, - \, \omega^{2} \, \mu \, \sum_{j=1}^{\aleph} \left( \epsilon_{0}(z_{j}) - \epsilon_{p}(\omega) \right) \, G_{k}(x,z_{j}) \cdot \int_{D_{j}}  E(y)  \, dy \\ &-& \, T_{2}(x)  \, + \, \mathcal{O}\left( a^{\frac{5-s}{2}} \; \left\Vert E \right\Vert_{\mathbb{L}^{2}(D)}\right), 
\end{eqnarray*}
where 
\begin{equation*}
    T_{2}(x) \, := \, \omega^{2} \, \mu \, \sum_{j=1}^{\aleph} \, \left( \epsilon_{0}(z_{j}) - \epsilon_{p}(\omega) \right) \, \int_{D_{j}} \int_{0}^{1} \nabla G_{k}(x,z_{j}+t(y-z_{j})) \cdot (y-z_{j}) \, dt \cdot E(y)  \, dy.
\end{equation*}
We estimate the above term as 
\begin{equation*}
    \left\vert T_{2}(x) \right\vert \,  \lesssim  \,   \sum_{j=1}^{\aleph} \left\vert  \epsilon_{0}(z_{j}) - \epsilon_{p}(\omega)  \right\vert \,  \left\vert \int_{D_{j}} \int_{0}^{1} \nabla G_{k}(x,z_{j}+t(y-z_{j})) \cdot (y-z_{j}) \, dt \cdot E(y)  \, dy \right\vert,
\end{equation*}
which, by using the fact that $\left\vert  \epsilon_{0}(z_{j}) - \epsilon_{p}(\omega)  \right\vert \, \sim \, 1$ and the H\"{o}lder inequality, gives us
\begin{eqnarray*}
    \left\vert T_{2}(x) \right\vert \,  & \lesssim & \,   \sum_{j=1}^{\aleph} \left\Vert  \int_{0}^{1} \nabla G_{k}(x,z_{j}+t(\cdot-z_{j})) \cdot (\cdot-z_{j}) \, dt \right\Vert_{\mathbb{L}^{2}(D_{j})} \, \left\Vert E \right\Vert_{\mathbb{L}^{2}(D_{j})} \\
    & \lesssim & \, \left(  \sum_{j=1}^{\aleph} \left\Vert  \int_{0}^{1} \nabla G_{k}(x,z_{j}+t(\cdot-z_{j})) \cdot (\cdot-z_{j}) \, dt \right\Vert^{2}_{\mathbb{L}^{2}(D_{j})} \right)^{\frac{1}{2}} \, \left\Vert E \right\Vert_{\mathbb{L}^{2}(D)}.
\end{eqnarray*}
Again, using the smoothness of $\nabla G_{k}(x, \cdot)$, for $x$ away from $D$, allows us to deduce that 
\begin{equation*}
     \left\vert T_{2}(x) \right\vert \, = \, \mathcal{O}\left( a^{\frac{5-s}{2}} \; \left\Vert E \right\Vert_{\mathbb{L}^{2}(D)}\right).
\end{equation*}
Then, 
\begin{eqnarray}\label{Equa345}
\nonumber
E^{s}(x)  \, - \, V^{s}(x) \, &=& \, - \, \omega^{2} \, \mu \, \sum_{j=1}^{\aleph} \left( \epsilon_{0}(z_{j}) - \epsilon_{p}(\omega) \right) \, G_{k}(x,z_{j}) \cdot \int_{D_{j}}  E(y)  \, dy  \, + \, \mathcal{O}\left( a^{\frac{5-s}{2}} \; \left\Vert E \right\Vert_{\mathbb{L}^{2}(D)}\right) \\
&\overset{(\ref{ape})}{=}& \, - \, \omega^{2} \, \mu \, \sum_{j=1}^{\aleph} \left( \epsilon_{0}(z_{j}) - \epsilon_{p}(\omega) \right) \, G_{k}(x,z_{j}) \cdot \int_{D_{j}}  E(y)  \, dy  \, + \, \mathcal{O}\left( a^{4 -h -s} \right).
\end{eqnarray}
In addition, by plugging $(\ref{BRTVMM+})$ into $(\ref{Equa345})$, we derive that
\begin{eqnarray}\label{Es=Vs+...}
\nonumber
E^{s}(x) \, - \, V^{s}(x) \, & = &  \, - \, \mu \, a^{3} \, \omega^{2} \, \sum_{j=1}^{\aleph} \frac{\epsilon_{0}(z_{j}) \, \left( \epsilon_{0}(z_{j}) - \epsilon_{p}(\omega) \right)}{\Lambda_{n_{0},j}(\omega)}  \, \langle V\left(z_{j},  \theta, q \right), \int_{B}  e_{n_{0}}^{(3)}(y)  \, dy \rangle \, G_{k}(x,z_{j}) \cdot \int_{B}  e_{n_{0}}^{(3)}(y)  \, dy  \\ &+& \mathcal{O}\left( a^{\min\left((3-s);(6-3t-2h-\frac{3s}{2});(4-h-s)\right)}\right),
\end{eqnarray}
where $\Lambda_{n_{0},j}(\omega)$ is a family of dispersion equations given by 
\begin{equation}\label{Eq26}
    \Lambda_{n_{0},j}(\omega) \, := \, \left( \epsilon_{0}(z_{j}) \, - \, \lambda_{n_{0}}^{(3)}(B) \, \left( \epsilon_{0}(z_{j}) \, - \,  \epsilon_{p}(\omega)\right) \right).
\end{equation}
Next, we solve the real part of the above dispersion equation. More precisely, we set $\omega_{P_{\ell},n_{0},j}$ to be solution to 
\begin{equation}\label{Solomegam}
    Re\left(\Lambda_{n_{0},j}(\omega_{P_{\ell},n_{0},j})\right) \, = \, \left( Re\left(\epsilon_{0}(z_{j})\right) \, - \, \lambda_{n_{0}}^{(3)}(B) \, \left( Re\left(\epsilon_{0}(z_{j})\right) - Re\left(\epsilon_{p}(\omega_{P_{\ell},n_{0},j})\right) \right)  \right) \, = \, 0,
\end{equation}
which gives us 
\begin{equation}\label{epsomegam}
    Re\left(\epsilon_{p}(\omega_{P_{\ell},n_{0},j})\right) \, = \, Re\left(\epsilon_{0}(z_{j})\right) \, \frac{\left(\lambda_{n_{0}}^{(3)}(B) \, - \, 1 \right)}{\lambda_{n_{0}}^{(3)}(B)}. 
\end{equation}
The existence and uniqueness of the solution $ \omega_{P_{\ell},n_{0},j}$ to $(\ref{Solomegam})$, or equivalently to $(\ref{epsomegam})$, has already been discussed in \cite[Lemma 5.7]{AG-MS-Maxwell}. We have, 
\begin{equation}\label{PlasmonicResonace}
    \omega^{2}_{P_{\ell},n_{0},j} \, = \, \frac{1}{2} \, \left[ 2 \, \omega^{2}_{0} \, - \, \gamma^{2} \, + \, \frac{\omega^{2}_{p} \, \lambda^{(3)}_{n_{0}}(B) \, \epsilon_{\infty}}{  \left[ \Re\left(\epsilon_{0}(z_{j})\right) \, \left(1 - \lambda^{(3)}_{n_{0}}(B) \right) + \lambda^{(3)}_{n_{0}}(B) \, \epsilon_{\infty} \right]} \, + \, \sqrt{\Delta} \right],
\end{equation}
where $\Delta$ is such that
\begin{eqnarray*}
\Delta &=& \left( \frac{\omega^{2}_{p} \, \lambda^{(3)}_{n_{0}}(B) \, \epsilon_{\infty}}{  \left[ \Re\left(\epsilon_{0}(z_{j})\right) \, \left(1 - \lambda^{(3)}_{n_{0}}(B) \right) + \lambda^{(3)}_{n_{0}}(B) \, \epsilon_{\infty} \right]} \right)^{2} \\ &-& \gamma^{2} \, \left( 4 \, \omega^{2}_{0} \, - \, \gamma^{2} \, + \, \frac{2 \, \omega^{2}_{p} \, \epsilon_{\infty} \, \lambda^{(3)}_{n_{0}}(B)}{\left[ \Re\left(\epsilon_{0}(z_{j})\right) \, \left(1 - \lambda^{(3)}_{n_{0}}(B) \right) + \lambda^{(3)}_{n_{0}}(B) \, \epsilon_{\infty} \right]} \right). 
\end{eqnarray*}
Furthermore, the frequency $\omega_{P_{\ell},n_{0},j}$ satisfies  
\begin{eqnarray}\label{epsomega=epsomegam+err}
\nonumber
    \Lambda_{n_{0},j}\left( \omega_{P_{\ell},n_{0},j} \right) \, &=& \, Im\left( \epsilon_{0}(z_{j}) \right) \, \left( 1 \, - \, \lambda^{(3)}_{n_{0}}(B) \right) \, + \, \lambda^{(3)}_{n_{0}}(B) \, Im\left( \epsilon_{p}(\omega_{P_{\ell},n_{0},j}) \right) \\ \nonumber
    \left\vert \Lambda_{n_{0},j}\left( \omega_{P_{\ell},n_{0},j} \right) \right\vert \, & \lesssim &  \left\Vert Im\left( \epsilon_{0} \right) \right\Vert_{\mathbb{L}^{\infty}(\Omega)}  \, + \, \left\vert Im\left( \epsilon_{p}(\omega_{P_{\ell},n_{0},j}) \right) \right\vert  \overset{(\ref{Lorentz-model})}{\underset{(\ref{ImEpsgamma})}{=}} \, \mathcal{O}\left( \gamma \right),% \\
    %& \lesssim & \left\Vert Im\left( \epsilon_{0} \right) \right\Vert_{\mathbb{L}^{\infty}(\Omega)}  \, + \,  \gamma  \, = \, \mathcal{O}\left( \gamma \right), 
\end{eqnarray}
%where the last equality is due to the fact that 
%\begin{equation}\label{ImEpsgamma}
% \left\Vert Im\left( \epsilon_{0} \right) \right\Vert_{\mathbb{L}^{\infty}(\Omega)} \, = \, \mathcal{O}\left( \gamma \right).    
%\end{equation}
Consequently, if we assume that
\begin{equation}\label{gammaah}
\gamma \, = \, \mathcal{O}\left( a^{h} \right),
\end{equation}
we derive from $(\ref{epsomega=epsomegam+err})$ the coming estimation 
\begin{equation*}
    \Lambda_{n_{0},j}\left( \omega_{P_{\ell},n_{0},j} \right) \, = \, \mathcal{O}\left( a^{h} \right). 
\end{equation*}
In addition, if we let $\omega$ such that 
\begin{equation}\label{Eq1238}
    \omega^{2}  \, = \, \omega^{2}_{P_{\ell},n_{0},j} \, \pm \, a^{h},   
\end{equation}
we get 
\begin{eqnarray*}
    \Lambda_{n_{0},j}\left( \omega \right) \, & \overset{(\ref{Eq26})}{=} & \,  \epsilon_{0}(z_{j}) \, - \, \lambda_{n_{0}}^{(3)}(B) \, \left( \epsilon_{0}(z_{j}) \, - \,  \epsilon_{p}(\omega)\right)  \\
    & = & \, Re\left( \epsilon_{0}(z_{j}) \right) \, - \, \lambda_{n_{0}}^{(3)}(B) \, Re\left( \epsilon_{0}(z_{j}) \, - \,  \epsilon_{p}(\omega_{P_{\ell},n_{0},j})  \right)  \, -  \, \lambda_{n_{0}}^{(3)}(B) \, Re\left(  \epsilon_{p}(\omega_{P_{\ell},n_{0},j}) - \epsilon_{p}(\omega) \right)  \\
    &+& \, i \, Im\left(\epsilon_{0}(z_{j})\right) \, - \, \lambda_{n_{0}}^{(3)}(B) \, i \, \left( Im\left( \epsilon_{0}(z_{j}) \right) \, - \, Im\left(  \epsilon_{p}(\omega)\right)\right)\\
    &\overset{(\ref{Solomegam})}{=}&  \, \lambda_{n_{0}}^{(3)}(B) \, Re\left(\epsilon_{p}(\omega) - \epsilon_{p}(\omega_{P_{\ell},n_{0},j}) \right)  \, + \, i \, Im\left(\epsilon_{0}(z_{j})\right) \, \left( 1 \, - \, \lambda_{n_{0}}^{(3)}(B) \right) \, + \, i \, \lambda_{n_{0}}^{(3)}(B) \, \, Im\left(\epsilon_{p}(\omega)\right).
\end{eqnarray*}
Hence, by taking the modulus, we obtain 
\begin{eqnarray*}
    \left\vert \Lambda_{n_{0},j}\left( \omega \right) \right\vert \, 
    &\lesssim &  \, \left\vert Re\left(\epsilon_{p}(\omega) - \epsilon_{p}(\omega_{P_{\ell},n_{0},j}) \right) \right\vert \, + \, \left\Vert Im\left(\epsilon_{0} \right) \right\Vert_{\mathbb{L}^{\infty}(\Omega)}  \, + \,  \left\vert Im\left(\epsilon_{p}(\omega)\right) \right\vert \\ 
    & \overset{(\ref{Lorentz-model})}{\underset{(\ref{ImEpsgamma}) } \lesssim} &\, \left\vert Re\left(\epsilon_{p}(\omega) - \epsilon_{p}(\omega_{P_{\ell},n_{0},j}) \right) \right\vert \, + \, \gamma \,
     \overset{(\ref{Lorentz-model})}{\lesssim} \,  \left\vert \omega^{2} \, - \, \omega^{2}_{P_{\ell},n_{0},j} \right\vert \, + \, \gamma,
\end{eqnarray*}
and, by using $(\ref{gammaah})$ and $(\ref{Eq1238})$, we deduce that  
\begin{equation*}
    \Lambda_{n_{0},j}\left( \omega \right) \, = \, \mathcal{O}\left( a^{h} \right), 
\end{equation*}
with $\omega$ satisfying $(\ref{Eq1238})$. Now, by returning to $(\ref{Es=Vs+...})$ and using  $(\ref{Eq26})$ we derive the following relation on the scattered fields
\begin{eqnarray}\label{Es-Vs}
\nonumber
E^{s}(x) - V^{s}(x)  \, & = & - \,  \mu \, a^{3} \, \omega^{2} \, \sum_{j=1}^{\aleph}  \, \frac{\epsilon_{0}(z_{j}) \, \left(\epsilon_{0}(z_{j}) \, - \, \Lambda_{n_{0},j}\left( \omega \right) \right)}{\lambda_{n_{0}}^{(3)}(B) \, \Lambda_{n_{0},j}\left( \omega \right)}  \, \langle V(z_{j}, \theta, q) , \int_{B}  e_{n_{0}}^{(3)}(y)  \, dy \rangle \, G_{k}(x,z_{j}) \cdot \int_{B}  e_{n_{0}}^{(3)}(y)  \, dy \\ &+& \mathcal{O}\left( a^{\min\left( (3-s), (6-3t-2h-\frac{3s}{2}), (4 - h - s) \right)}\right).
\end{eqnarray}
This proves $(\ref{Es-Vs-intro})$. For the L.H.S of $(\ref{Es-Vs})$, in order to get the corresponding far-field expressions, we return to $(\ref{LFI})$ that
\begin{equation}\label{LFI2}
E^{s}(x) - V^{s}(x) \, = \, - \, \omega^{2} \, \mu \, \sum_{j=1}^{\aleph} \int_{D_{j}} G_{k}(x,y) \cdot E(y) \, \left( \epsilon_{0}(y) - \epsilon_{p}(\omega) \right) \, dy.
\end{equation}
Furthermore, thanks to \cite[Theorem 6.9]{colton2019inverse}, we know that the following asymptotic relation holds 
\begin{equation}\label{LFI3}
    G_{k}(x,y) \, = \, \frac{e^{i \, k \, \left\vert x \right\vert}}{\left\vert x \right\vert}  \, \left( G_{k}^{\infty}(\hat{x},y) \, + \, \mathcal{O}\left( \frac{1}{\left\vert x \right\vert} \right)  \right), \quad \left\vert x \right\vert \rightarrow + \infty.
\end{equation}
Then, by plugging $(\ref{LFI3})$ into $(\ref{LFI2})$, we obtain  
\begin{eqnarray*}
E^{s}(x) - V^{s}(x) \, &=& \frac{e^{i \, k \, \left\vert x \right\vert}}{\left\vert x \right\vert} \, \Bigg[  - \, \omega^{2} \, \mu \, \sum_{j=1}^{\aleph} \int_{D_{j}} G_{k}^{\infty}(\hat{x},y) \cdot E(y) \, \left( \epsilon_{0}(y) - \epsilon_{p}(\omega) \right) \, dy \\
&+& \mathcal{O}\left( \frac{1}{\left\vert x \right\vert} \sum_{j=1}^{\aleph} \int_{D_{j}}  E(y) \, \left( \epsilon_{0}(y) - \epsilon_{p}(\omega) \right) \, dy \right) \Bigg]. 
\end{eqnarray*}
By estimating the second term on the R.H.S, we obtain 
\begin{eqnarray*}
    T_{0}(x) \, &:=& \, \frac{1}{\left\vert x \right\vert} \sum_{j=1}^{\aleph} \int_{D_{j}}  E(y) \, \left( \epsilon_{0}(y) - \epsilon_{p}(\omega) \right) \, dy \\ 
    \left\vert T_{0}(x) \right\vert \, & \leq & \, \frac{1}{\left\vert x \right\vert} \sum_{j=1}^{\aleph} \left\vert  \int_{D_{j}}  E(y) \, \left( \epsilon_{0}(y) - \epsilon_{p}(\omega) \right) \, dy \right\vert \\
    & \leq & \, \frac{1}{\left\vert x \right\vert} \sum_{j=1}^{\aleph} \left\Vert E \right\Vert_{\mathbb{L}^{2}(D_{j})} \; \left\Vert  \epsilon_{0}(\cdot) \, - \, \epsilon_{p}(\omega)  \right\Vert_{\mathbb{L}^{\infty}(D_{j})} \, \left\vert D_{j} \right\vert^{\frac{1}{2}},
    \end{eqnarray*}
    which, by knowing that $\left\Vert  \epsilon_{0}(\cdot) \, - \, \epsilon_{p}(\omega)  \right\Vert_{\mathbb{L}^{\infty}(D_{j})} \, \sim \, 1$, gives 
    \begin{equation*}
    \left\vert T_{0}(x) \right\vert \,  \lesssim  \, \frac{1}{\left\vert x \right\vert} \sum_{j=1}^{\aleph} \left\Vert E \right\Vert_{\mathbb{L}^{2}(D_{j})} \; \left\vert D_{j} \right\vert^{\frac{1}{2}} 
     \lesssim  \, \frac{1}{\left\vert x \right\vert} \, \left\Vert E \right\Vert_{\mathbb{L}^{2}(D)} \; \left\vert D \right\vert^{\frac{1}{2}} 
     \overset{(\ref{ape})}{=}  \mathcal{O}\left( \frac{a^{3-h-s}}{\left\vert x \right\vert} \right).
\end{equation*}
Hence, 
\begin{equation*}
E^{s}(x) - V^{s}(x) \, = \, \frac{e^{i \, k \, \left\vert x \right\vert}}{\left\vert x \right\vert} \, \left[  - \, \omega^{2} \, \mu \, \sum_{j=1}^{\aleph} \int_{D_{j}} G_{k}^{\infty}(\hat{x},y) \cdot E(y) \, \left( \epsilon_{0}(y) - \epsilon_{p}(\omega) \right) \, dy \, 
+ \, \mathcal{O}\left( \frac{a^{3-h-s}}{\left\vert x \right\vert} \right) \right]. 
\end{equation*}
This implies, 
\begin{equation*}
E^{\infty}(\hat{x}) - V^{\infty}(\hat{x}) \, = \,   - \, \omega^{2} \, \mu \, \sum_{j=1}^{\aleph} \int_{D_{j}} G_{k}^{\infty}(\hat{x},y) \cdot E(y) \, \left( \epsilon_{0}(y) - \epsilon_{p}(\omega) \right) \, dy. 
\end{equation*}
Furthermore, by repeating the same arguments used to derive $(\ref{Es-Vs})$ and letting $\left\vert x \right\vert \gg 1$ in $(\ref{Es-Vs})$, we obtain 
\begin{eqnarray*}
E^{\infty}(\hat{x}) - V^{\infty}(\hat{x})  \, & = & - \,  \mu \, a^{3} \, \omega^{2} \, \sum_{j=1}^{\aleph}  \, \frac{\epsilon_{0}(z_{j}) \, \left(\epsilon_{0}(z_{j}) \, - \, \Lambda_{n_{0},j}\left( \omega \right) \right)}{\lambda_{n_{0}}^{(3)}(B) \, \Lambda_{n_{0},j}\left( \omega \right)}  \, \langle V(z_{j}, \theta, q ), \int_{B}  e_{n_{0}}^{(3)}(y)  \, dy \rangle \,  G^{\infty}_{k}(\hat{x},z_{j}) \cdot \int_{B}  e_{n_{0}}^{(3)}(y)  \, dy \\ &+& \mathcal{O}\left( a^{\min\left( (3-s), (6-3t-2h-\frac{3s}{2}), (4 - h - s) \right)}\right) ,
\end{eqnarray*}
which by using $(\ref{Eq1238})$ further indicates, 
\begin{eqnarray*}
&& E^{\infty}(\hat{x}) - V^{\infty}(\hat{x})=  \\ &  & - \,  \mu \, a^{3}  \, \sum_{j=1}^{\aleph}  \frac{ \omega^{2}_{P_{\ell},n_{0},j} \, \epsilon_{0}(z_{j}) \, \left(\epsilon_{0}(z_{j}) \, - \, \Lambda_{n_{0},j}\left( \omega \right) \right)}{\lambda_{n_{0}}^{(3)}(B) \, \Lambda_{n_{0},j}\left( \omega \right)}  \, \langle V(z_{j},  \theta, q), \int_{B}  e_{n_{0}}^{(3)}(y)  \, dy \rangle \,  G^{\infty}_{k}(\hat{x},z_{j}) \cdot \int_{B}  e_{n_{0}}^{(3)}(y)  \, dy  \\ &+& \mathcal{O}\left( a^{\min\left( (3-s), (6-3t-2h-\frac{3s}{2}), (4 - 2h - s) \right)}\right),
\end{eqnarray*}
with $\hat{x} \in \mathbb{S}^{2}$ and $G^{\infty}_{k}(\cdot,z)$ being the far-field associated to the Green's kernel solution to $(\ref{Eq0401})$. Besides, we have 
\begin{eqnarray*}
    \frac{\epsilon_{0}(z_{j}) \, \left(\epsilon_{0}(z_{j}) \, - \, \Lambda_{n_{0},j}\left( \omega \right) \right)}{\lambda_{n_{0}}^{(3)}(B) \, \Lambda_{n_{0},j}\left( \omega \right)} &=&     \frac{\epsilon_{0}(z_{j}) \, \left(\epsilon_{0}(z_{j}) \, - \, \Lambda_{n_{0},j}\left( \omega_{P_{\ell},n_{0},j} \right) \right)}{\lambda_{n_{0}}^{(3)}(B) \, \Lambda_{n_{0},j}\left( \omega \right)} 
    +     \frac{\epsilon_{0}(z_{j}) \, \left(\Lambda_{n_{0},j}\left( \omega_{P_{\ell},n_{0},j} \right) \, - \, \Lambda_{n_{0},j}\left( \omega \right) \right)}{\lambda_{n_{0}}^{(3)}(B) \, \Lambda_{n_{0},j}\left( \omega \right)} \\
     &=&     \frac{\epsilon_{0}(z_{j}) \, \left(\epsilon_{0}(z_{j}) \, - \, \Lambda_{n_{0},j}\left( \omega_{P_{\ell},n_{0},j} \right) \right)}{\lambda_{n_{0}}^{(3)}(B) \, \Lambda_{n_{0},j}\left( \omega \right)} \, + \, \mathcal{O}\left( 1 \right).
\end{eqnarray*}
Then, 
\begin{eqnarray}\label{BMER}
\nonumber
&& E^{\infty}(\hat{x}) - V^{\infty}(\hat{x}) = \\ \nonumber &  & - \,  \mu \, a^{3} \,  \sum_{j=1}^{\aleph}  \,  \frac{\omega^{2}_{P_{\ell},n_{0},j} \, \epsilon_{0}(z_{j}) \, \left(\epsilon_{0}(z_{j}) \, - \, \Lambda_{n_{0},j}\left( \omega_{P_{\ell},n_{0},j} \right) \right)}{\lambda_{n_{0}}^{(3)}(B) \, \Lambda_{n_{0},j}\left( \omega \right)}  \, \langle V(z_{j}, \theta, q), \int_{B}  e_{n_{0}}^{(3)}(y)  \, dy \rangle \, G^{\infty}_{k}(\hat{x},z_{j}) \cdot \int_{B}  e_{n_{0}}^{(3)}(y)  \, dy  \\ && \qquad + \mathcal{O}\left( a^{\min\left( (3-s), (6-3t-2h-\frac{3s}{2}), (4 - 2h - s) \right)}\right).
\end{eqnarray}
\medskip
\newline 
The kernel $G^{\infty}_{k}(\cdot, \cdot)$ is unknown on the right hand side of $(\ref{BMER})$  since the kernel $G_{k}(\cdot,\cdot)$ is unknown. We need to substitute it with a known (measured) term. To achieve this, we propose setting the following Mixed Electromagnetic Reciprocity Relation.
\medskip
\begin{proposition}\label{PropMER} 
The following Mixed Electromagnetic Reciprocity Relation holds, 
\begin{equation}\label{MERFormula}
         \left(G_{k}^{\infty}\right)^{\top}( \hat{x},z) \cdot \left( \hat{x} \times q \right) \, = \, -\frac{1}{4 \, \pi} \, V(z, - \, \hat{x},q), \quad z \in \mathbb{R}^{3} \quad \text{and} \quad \hat{x}, q \in \mathbb{S}^{2}.
\end{equation}
\end{proposition}
\begin{proof}
See Subsection \ref{P21}.
\end{proof}
\bigskip
Hence, by taking the inner product on each side of the equation given by $(\ref{BMER})$ with the vector $\left( \hat{x} \times q \right)$, we obtain
\begin{eqnarray*}
\nonumber
&& \langle \left( E^{\infty} \, - \, V^{\infty}\right)(\hat{x}) , \left( \hat{x} \times q \right) \rangle \\ & = & - \,  \mu \, a^{3} \,  \sum_{j=1}^{\aleph}  \, \frac{\omega^{2}_{P_{\ell},n_{0},j} \, \epsilon_{0}(z_{j}) \, \left(\epsilon_{0}(z_{j}) \, - \, \Lambda_{n_{0},j}\left( \omega_{P_{\ell},n_{0},j} \right) \right)}{\lambda_{n_{0}}^{(3)}(B) \, \Lambda_{n_{0},j}\left( \omega \right)}  \, \langle V(z_{j}, \theta, q), \int_{B}  e_{n_{0}}^{(3)}(y)  \, dy \rangle \, \langle G^{\infty}_{k}(\hat{x},z_{j}) \cdot \int_{B}  e_{n_{0}}^{(3)}(y)  \, dy , \left( \hat{x} \times q \right) \rangle  \\ &+& \mathcal{O}\left( a^{\min\left( (3-s), (6-3t-2h-\frac{3s}{2}), (4 - h - s) \right)}\right),
\end{eqnarray*}
or, equivalently\footnote{We have used the fact that for two arbitrary vector $A$ and $B$, and any arbitrary matrix $M$, we have 
\begin{equation*}
    \langle M \cdot A, B \rangle \, = \, \langle A , M^{\top} \cdot B \rangle.  
\end{equation*}
}, 
\begin{eqnarray*}
\nonumber
&& \langle \left( E^{\infty} \, - \, V^{\infty}\right)(\hat{x}) , \left( \hat{x} \times q \right) \rangle \\ & = & - \,  \mu \, a^{3} \,  \sum_{j=1}^{\aleph}  \, \frac{\omega^{2}_{P_{\ell},n_{0},j} \, \epsilon_{0}(z_{j}) \, \left(\epsilon_{0}(z_{j}) \, - \, \Lambda_{n_{0},j}\left( \omega_{P_{\ell},n_{0},j} \right) \right)}{\lambda_{n_{0}}^{(3)}(B) \, \Lambda_{n_{0},j}\left( \omega \right)}  \, \langle V(z_{j},  \theta, q), \int_{B}  e_{n_{0}}^{(3)}(y)  \, dy \rangle \, \langle  \int_{B}  e_{n_{0}}^{(3)}(y)  \, dy , \left( G^{\infty}_{k}\right)^{\top}(\hat{x},z_{j}) \cdot \left( \hat{x} \times q \right) \rangle  \\ &+& \mathcal{O}\left( a^{\min\left( (3-s), (6-3t-2h-\frac{3s}{2}), (4 - h - s) \right)}\right). 
\end{eqnarray*}
Thanks to the Mixed Electromagnetic Reciprocity Relation in Proposition \ref{PropMER}, formula $(\ref{MERFormula})$, the above equation becomes
\begin{eqnarray*}
\nonumber
&& \langle \left( E^{\infty} \, - \, V^{\infty}\right)(\hat{x}) , \left( \hat{x} \times q \right) \rangle \\ \nonumber & = &  \frac{ \mu \, a^{3}}{4 \, \pi}  \, \sum_{j=1}^{\aleph}  \, \frac{\omega^{2}_{P_{\ell},n_{0},j} \, \epsilon_{0}(z_{j}) \, \left(\epsilon_{0}(z_{j}) \, - \, \Lambda_{n_{0},j}\left( \omega_{P_{\ell},n_{0},j} \right) \right)}{\lambda_{n_{0}}^{(3)}(B) \, \Lambda_{n_{0},j}\left( \omega \right)}  \, \langle V(z_{j},  \theta, q), \int_{B}  e_{n_{0}}^{(3)}(y)  \, dy \rangle \, \langle V\left(z_{j}, - \hat{x}, q \right), \int_{B}  e_{n_{0}}^{(3)}(y)  \, dy  \rangle  \\ &+& \mathcal{O}\left( a^{\min\left( (3-s), (6-3t-2h-\frac{3s}{2}), (4 - h - s) \right)}\right), 
\end{eqnarray*}
which proves $(\ref{Es-Vs-intro-far-field})$. In the particular case for the back-scattered direction, i.e. $ \hat{x} \, = \, - \, \theta$, we deduce 
\begin{eqnarray*}
\nonumber
\langle \left( E^{\infty} \, - \, V^{\infty}\right)(- \, \theta) , \left( \theta \times q \right) \rangle  &=&  - \frac{\mu \, a^{3}}{4 \, \pi}  \, \sum_{j=1}^{\aleph}  \, \frac{\omega^{2}_{P_{\ell},n_{0},j} \, \epsilon_{0}(z_{j}) \, \left(\epsilon_{0}(z_{j}) \, - \, \Lambda_{n_{0},j}\left( \omega_{P_{\ell},n_{0},j} \right) \right)}{\lambda_{n_{0}}^{(3)}(B) \, \Lambda_{n_{0},j}\left( \omega \right)} \, \left(  \langle V\left(z_{j},  \theta, q \right), \int_{B}  e_{n_{0}}^{(3)}(y)  \, dy \rangle \right)^{2} \\ &+& \mathcal{O}\left( a^{\min\left( (3-s), (6-3t-2h-\frac{3s}{2}), (4 - h - s) \right)}\right). 
\end{eqnarray*}
This ends the proof Theorem \ref{ThmResult}.
\section{Appendix}\label{SectionIV}
This section will be divided into four parts. The first subsection is devoted to demonstrate the Mixed Electromagnetic Reciprocity Relation that was announced in Proposition \ref{PropMER}. In the second subsection, we prove an a-priori estimate related to the electric field that is used in Section \ref{SecPrfThm}. In the third subsection, we give sense of the formula $(\ref{Equa0805})$. The last subsection is dedicated to justifying formula $(\ref{GT=G})$, which is related to the Green's kernel $G_{k}(\cdot, \cdot)$ and used in the proof of the Mixed Electromagnetic Reciprocity Relation.   
\subsection{Proof Proposition \ref{PropMER} (Mixed Electromagnetic Reciprocity Relation)}\label{P21}
We start by recalling, from $(\ref{Eq0401})$, that the Green's kernel associated to heterogeneous medium is the solution to 
    \begin{equation}\label{SEq}
    \nabla \times \nabla \times \left(G_{k}(\cdot,z)\right) \, - \, k^{2}(\cdot) \, G_{k}(\cdot,z) \, = \, \delta\left(\cdot, z \right) \, I_{3}, \quad   \text{in} \quad \mathbb{R}^{3}, 
    \end{equation}
    where $k^{2}(\cdot) \, = \, k^{2}$ in $\mathbb{R}^{3} \setminus \overline{\Omega}$.
    The corresponding Green's kernel to the homogeneous medium is given by
    \begin{equation}\label{Curl2Pik}
        \nabla \times \nabla \times \left(\Pi_{k}(\cdot,z)\right) \, - \, k^{2} \, \Pi_{k}(\cdot,z) \, = \, \delta\left(\cdot, z \right) \, I_{3}, \quad  \text{in} \quad \mathbb{R}^{3}.
    \end{equation}
 It is known that 
\begin{equation}\label{KernelP}
    \Pi_{k}(x,z) \, = \, \frac{1}{k^{2}} \, \underset{x}{\nabla} \, \underset{x}{\nabla} \, \Phi_{k}(x,z) \, + \, \Phi_{k}(x,z) \, I_{3}, \quad x \neq z,  
\end{equation}
where $\Phi_{k}(\cdot, \cdot)$ is the fundamental solution to the Helmholtz equation given by $(\ref{FSHE})$. Subtracting $(\ref{Curl2Pik})$ from $(\ref{SEq})$, it yields
\begin{equation}\label{Equa0330}
    \nabla \times \nabla \times \left(\left( G_{k} \, - \, \Pi_{k} \right)(\cdot,z)\right) \, - \, k^{2} \, \left( G_{k} \, - \, \Pi_{k} \right)(\cdot,z) \, = \, \left( k^{2}(\cdot) \, - \, k^{2} \right) \,  G_{k}(\cdot,z), \quad \text{in} \quad \mathbb{R}^{3}, 
\end{equation}
and its solution is given by  
\begin{equation}\label{Equa0805}
    \left( G_{k} \, - \, \Pi_{k} \right)(x,z) \, = \, \int_{\Omega} \Pi_{k}(x,y) \cdot \left( k^{2}(y) \, - \, k^{2} \right) \, G_{k}(y,z) \, dy, \quad x \in \mathbb{R}^{3},
\end{equation}
which will be provided later, see Subsection \ref{SubSecIIEqua242}. $\Omega$ represents the support of the right hand side of $(\ref{Equa0330})$. Hence, by taking its transpose and using the fact that $\Pi_{k}^{\top}(\cdot,\cdot) \, = \, \Pi_{k}(\cdot,\cdot)$,
%\begin{equation*}
%\Pi_{k}^{\top}(\cdot,\cdot) \, = \, \Pi_{k}(\cdot,\cdot) \quad \text{and} \quad \overline{\Pi_{k}}(\cdot,\cdot) \, = \, \Pi_{-k}(\cdot,\cdot),
%\end{equation*}
we obtain\footnote{The motivation behind this is $(\ref{GT=G})$.} 
\begin{equation}\label{AddEquaI}
    \left( G_{k}^{\top} \, - \, \Pi_{k} \right)(x,z) \, = \, \int_{\Omega} G_{k}^{\top}(y,z) \cdot \left( k^{2}(y) \, - \, k^{2} \right) \, \Pi_{k}(x,y)   \, dy, \quad x \in \mathbb{R}^{3}.
\end{equation}
 The corresponding far-field to $(\ref{AddEquaI})$ will be given by 
\begin{equation}\label{Equa0807}
    \left( G_{k}^{\top} \, - \, \Pi_{k} \right)^{\infty}(\hat{x},z) \, = \, \int_{\Omega} G_{k}^{\top}(y,z) \cdot \left( k^{2}(y) \, - \, k^{2} \right) \, \Pi_{k}^{\infty}(\hat{x},y) \, dy, \quad x \in \mathbb{R}^{3}.
\end{equation}
Moreover, from $(\ref{KernelP})$, we know that 
\begin{equation}\label{Equa0814}
    \Pi_{k}^{\infty}(\hat{x},y) \, = \, \frac{e^{- \, i \, k \, \hat{x} \cdot y}}{4 \, \pi} \, \left( I \, - \, \hat{x} \otimes \hat{x} \right).
\end{equation}
Then, by plugging the above expression into $(\ref{Equa0807})$, we obtain 
\begin{equation*}
    \left( G_{k}^{\top} \, - \, \Pi_{k} \right)^{\infty}(\hat{x},z) \, = \, \hat{x} \times \left( \int_{\Omega} \frac{e^{- \, i \, k \, \hat{x} \cdot y}}{4 \, \pi} \, \left( k^{2}(y) \, - \, k^{2} \right) \, G_{k}^{\top}(y,z) \, dy \times \hat{x} \right).
\end{equation*}
Besides, by multiplying each side with the vector $(\hat{x} \times q)$, we obtain\footnote{We have used the fact that $\left(G_{k}^{\top}\right)^{\infty} \, = \, \left(G_{k}^{\infty}\right)^{\top}$.}
\begin{eqnarray}\label{SEq1}
\nonumber
    \left(G_{k}^{\infty}\right)^{\top}(\hat{x},z) \cdot (\hat{x} \times q) \, &=& \, \Pi_{k}^{\infty}(\hat{x},z) \cdot (\hat{x} \times q) \, + \,   \int_{\Omega} \frac{e^{- \, i \, k \, \hat{x} \cdot y}}{4 \, \pi} \, \left( k^{2}(y) \, - \, k^{2} \right) \, G_{k}^{\top}(y,z) \cdot (\hat{x} \times q) \, dy \\ 
    &\overset{(\ref{Equa0814})}{=}& \, \frac{e^{- \, i \, k \, \hat{x} \cdot z}}{4 \, \pi} \, \left( \hat{x} \times q \right) \, + \,   \int_{\Omega} \frac{e^{- \, i \, k \, \hat{x} \cdot y}}{4 \, \pi} \, \left( k^{2}(y) \, - \, k^{2} \right) \, G_{k}^{\top}(y,z) \cdot (\hat{x} \times q) \, dy.
\end{eqnarray}
Now, from $(\ref{DefEinc})$, we recall that 
\begin{equation*}
    V^{Inc}\left( z, d, q \right) \, = \, \left( d \times q \right) \, e^{i \, k \, z \cdot d}, \quad \text{for} \quad d , q \in \mathbb{S}^{2}, 
\end{equation*}
we deduce, from $(\ref{SEq1})$, the following equation 
\begin{equation}\label{Equa0817}
    \left( G_{k}^{\infty} \right)^{\top}(\hat{x},z) \cdot \left( \hat{x} \times q \right) \, = \, - \frac{1}{4 \, \pi} \, V^{Inc}(z, - \hat{x}, q) \, - \, \frac{1}{4 \, \pi}  \int_{\Omega}  \left( k^{2}(y) \, - \, k^{2} \right) \, G_{k}^{\top}(y,z) \cdot V^{Inc}(y, - \hat{x}, q) \, dy.
\end{equation}
The total field equation should be recalled to analyze the second term on the R.H.S of the above equation. We have,   
\begin{equation*}
    \underset{z}{\nabla} \times \underset{z}{\nabla} \times \left( V(z,  \hat{x}, q) \right) \, - \, k^{2}(z) \, V(z,  \hat{x}, q) \, = \, 0, \quad z \in \mathbb{R}^{3} \quad \text{and} \quad \hat{x}, q \in \mathbb{S}^{2}, 
\end{equation*}
and, by knowing that $V(\cdot, \cdot, \cdot) \, = \, V^{s}(\cdot, \cdot, \cdot) \, + \, V^{Inc}(\cdot, \cdot, \cdot)$ with 
\begin{equation*}
    \underset{z}{\nabla} \times \underset{z}{\nabla} \times(V^{Inc}(z,  \hat{x}, q)) \, - \, k^{2} \, V^{Inc}(z, \hat{x}, q) \, = \, 0, \quad z \in \mathbb{R}^{3}, 
\end{equation*}
we deduce
\begin{equation*}\label{Vsx}
    \underset{z}{\nabla} \times \underset{z}{\nabla} \times\left( V^{s}(z,\hat{x}, q) \right) \, - \, k^{2}(z) \, V^{s}(z,\hat{x}, q) \, = \, \left(k^{2}(z) \, - \, k^{2} \right) \, V^{Inc}(z, \hat{x}, q).  
\end{equation*}
The solution to the above equation will be given by 
\begin{equation}\label{Equa0826}
    V^{s}(z, \hat{x}, q) \, = \, \int_{\Omega} G_{k}(z,y) \, \cdot \left( k^{2}(y) \, - \, k^{2} \right) \, V^{Inc}(y,  \hat{x}, q) \, dy.
\end{equation}
Then, as we have 
\begin{equation}\label{GT=G}
    G_{k}(z,y) \; = \; G^{\top}_{k}(y,z), \quad y \neq z,    
\end{equation}
see Subsection \ref{SubsectionGT=G}, we rewrite $(\ref{Equa0826})$ as 
\begin{equation}\label{Equa0828}
    V^{s}(z, \hat{x}, q) \, = \, \int_{\Omega} G_{k}^{\top}(y,z) \cdot \left( k^{2}(y) \, - \, k^{2} \right) \, V^{Inc}(y,  \hat{x}, q) \, dy.
\end{equation}
By gathering $(\ref{Equa0817})$ and $(\ref{Equa0828})$ we end up with the following equation 
\begin{equation*}
    \left( G_{k}^{\infty} \right)^{\top}(\hat{x},z) \cdot \left( \hat{x} \times q \right) \, = \, - \, \frac{1}{4 \, \pi} \, V^{Inc}(z, -\hat{x}, q) \, - \, \frac{1}{4 \, \pi} \, V^{s}(z, -\hat{x}, q) \, = \, - \, \frac{1}{4 \, \pi} \, V(z, -\hat{x}, q).
\end{equation*}
This proves $(\ref{MERFormula})$ and justifies the proof of Proposition \ref{PropMER}.
\medskip
\newline 
We conclude this subsection by noting the following remark.
\begin{remark}
    The result proved in this subsection is more general than the one proved in \cite[Theorem 5]{potthast1998point}, for the case of an exterior domain with specific boundary conditions. 
\end{remark}
\subsection{A-priori estimate}\label{Subsection42}
\begin{proposition}\label{Propapest}
Under the condition 
\begin{equation*}
    3 \, - \, 3 \, t \, - s \, - \, h \, > \, 0,  
\end{equation*}
    we have, 
    \begin{equation}\label{Equa1359}
    \left\Vert u_{1} \right\Vert_{\mathbb{L}^{2}(D)}  \,  \lesssim  \,   a^{-h} \,  \left\Vert u_{0} \right\Vert_{\mathbb{L}^{2}(D)}.
\end{equation}
where the parameter $h$ is such that $0 < h < 1$, the parameter $s$ fulfills $\aleph \sim a^{-s}$ and the parameter $t$ is such that $d \sim a^{t}$. 
\end{proposition}
\begin{proof}
We assume that $D \, =  \, \overset{\aleph}{\underset{j=1} \cup}D_{j}$ with $\aleph \, \sim \, a^{-s}$. To investigate Proposition \ref{Propapest}, we recall, from $(\ref{SK})$, that the total field $u_{1}(\cdot)$ is the solution to 
\begin{equation*}
u_{1}(x) + \omega^{2} \, \mu \, \int_{D} G_{k}(x,y) \cdot u_{1}(y) \, (\epsilon_{0}(y)-\epsilon_{p}(\omega)) \, dy = u_{0}(x), \quad x \in \mathbb{R}^{3}.  
\end{equation*}
Now, we let $x \in D_{m}$ and we rewrite the above equation as 
\begin{eqnarray}\label{SK-appendix}
\nonumber
u_{1}(x) \, &+& \, \omega^{2} \, \mu \, \int_{D_{m}} G_{k}(x,y) \cdot u_{1}(y) \, (\epsilon_{0}(y)-\epsilon_{p}(\omega)) \, dy \\
&+& \, \omega^{2} \, \mu \, \sum_{j=1 \atop j \neq m}^{\aleph} \int_{D_{j}} G_{k}(x,y) \cdot u_{1}(y) \, (\epsilon_{0}(y)-\epsilon_{p}(\omega)) \, dy \, = \, u_{0}(x), \quad x \in D_{m}.  
\end{eqnarray}
In addition, thanks to \cite[Theorem 2.1]{AG-MS-Maxwell}, we know that 
\begin{equation}\label{GkDecomposition}
    G_{k}(x,y) \, = \, \frac{1}{\omega^{2} \, \mu \, \epsilon_{0}(y)} \, \underset{y}{\nabla} \, \underset{y}{\nabla} \Phi_{0}(x,y) \, + \, \Gamma(x,y), \qquad x \neq y. 
\end{equation}
Then, by plugging $(\ref{GkDecomposition})$ into $(\ref{SK-appendix})$, we obtain 
\begin{eqnarray}\label{Equa1134}
\nonumber
u_{1}(x) \, &-&  \, \nabla M_{D_{m}}\left( u_{1}(\cdot) \, \tau(\cdot) \right)(x) \, + \, \omega^{2} \, \mu \, \sum_{j=1 \atop j \neq m}^{\aleph} \int_{D_{j}} G_{k}(x,y) \cdot u_{1}(y) \, (\epsilon_{0}(y)-\epsilon_{p}(\omega)) \, dy \\ &=& \, u_{0}(x) \, - \, \omega^{2} \, \mu \, \int_{D_{m}} \Gamma(x,y) \cdot u_{1}(y) \, (\epsilon_{0}(y)-\epsilon_{p}(\omega)) \, dy, \qquad x \in D_{m},  
\end{eqnarray}
where
\begin{equation*}
\tau(y) \, := \, \frac{(\epsilon_{0}(y) -\epsilon_{p}(\omega))}{\epsilon_{0}(y)},  
\end{equation*}
and $\nabla M_{D_{m}}\left( \cdot \right)$ is the Magnetization operator defined by $(\ref{DefNDefM})$. Furthermore, using Taylor expansion, the equation $(\ref{Equa1134})$ becomes,  
\begin{eqnarray*}
\left[  I \, - \, \tau(z_{m}) \, \nabla M_{D_{m}}\right]\left( u_{1}\right)(x) \, &=& \, u_{0}(x) \, - \, \omega^{2} \, \mu \,  \sum_{j=1 \atop j \neq m}^{\aleph} (\epsilon_{0}(z_{j}) \, - \, \epsilon_{p}(\omega)) \,  G_{k}(z_{m},z_{j}) \cdot \int_{D_{j}} u_{1}(y) \,  \, dy \\
&-& \, \omega^{2} \, \mu \,  \sum_{j=1 \atop j \neq m}^{\aleph} (\epsilon_{0}(z_{j}) \, - \, \epsilon_{p}(\omega)) \, \int_{0}^{1} \nabla_{x}G_{k}(z_{m}+t(x-z_{m}),z_{j}) \cdot (x - z_{m}) \, dt \cdot \int_{D_{j}} u_{1}(y) \,  \, dy \\
&-& \, \omega^{2} \, \mu \,  \sum_{j=1 \atop j \neq m}^{\aleph} (\epsilon_{0}(z_{j}) \, - \, \epsilon_{p}(\omega)) \, \int_{D_{j}} \int_{0}^{1} \nabla_{y} G_{k}(z_m,z_{j}+t(y-z_{j})) \cdot (y - z_{j}) \, dt \cdot u_{1}(y) \,  \, dy \\
 &-& \, \omega^{2} \, \mu \, \sum_{j=1 \atop j \neq m}^{\aleph} \int_{D_{j}} G_{k}(x,y) \cdot u_{1}(y) \, \int_{0}^{1} \nabla \epsilon_{0}(z_{m}+t(y-z_{m})) \cdot (y - z_{m}) \, dt \, dy \\ &+& \, \nabla M_{D_{m}}\left( u_{1}(\cdot) \, \int_{0}^{1} \nabla \tau(z_{m}+t(\cdot-z_{m})) \cdot (\cdot - z_{m}) \, dt \right)(x) \\ &-& \, \omega^{2} \, \mu \, \int_{D_{m}} \Gamma(x,y) \cdot u_{1}(y) \, (\epsilon_{0}(y)-\epsilon_{p}(\omega)) \, dy, \qquad x \in D_{m}.
\end{eqnarray*}
Besides, by scaling the above equation from the domain $D_{m}$ to the domain $B$, i.e. we let $x \, = \, z_{m} \, + a \, \eta$ and $y \, = \, z_{j} \, + a \, \xi$ with $\eta, \, \xi \, \in \, B \, \subset\, B(0,1)$, we obtain 
\begin{eqnarray*}
\left[  I \, - \, \tau(z_{m}) \, \nabla M_{B}\right]\left( \tilde{u}_{1,m}\right)(\eta) \, &=& \, \tilde{u}_{0,m}(\eta) \, - \, \omega^{2} \, \mu \, a^{3} \, \sum_{j=1 \atop j \neq m}^{\aleph} (\epsilon_{0}(z_{j}) \, - \, \epsilon_{p}(\omega)) \,  G_{k}(z_{m},z_{j}) \cdot \int_{B} \tilde{u}_{1,j}(\xi) \, d\xi \\
&-& \, \omega^{2} \, \mu \, a^{4} \, \sum_{j=1 \atop j \neq m}^{\aleph} (\epsilon_{0}(z_{j}) \, - \, \epsilon_{p}(\omega)) \, \int_{0}^{1} \nabla_{\eta}G_{k}(z_{m} \, + \, t \, a \eta, z_{j}) \cdot \eta \, dt \cdot \int_{B} \tilde{u}_{1,j}(\xi)  \, d\xi \\
&-& \, \omega^{2} \, \mu \, a^{4} \, \sum_{j=1 \atop j \neq m}^{\aleph} (\epsilon_{0}(z_{j}) \, - \, \epsilon_{p}(\omega)) \, \int_{B} \int_{0}^{1} \nabla_{\xi} G_{k}(z_{m},\, z_{j}+t \, a \, \xi) \cdot \xi \, dt \cdot \tilde{u}_{1,j}(\xi)   \, d\xi \\
 &-& \, \omega^{2} \, \mu \, a^{4} \, \sum_{j=1 \atop j \neq m}^{\aleph} \int_{B} G_{k}(z_{m} \, + \, a \, \eta , z_{j} \, + \, a \, \xi ) \cdot \tilde{u}_{1,j}(\xi) \, \int_{0}^{1} \nabla \epsilon_{0}(z_{m} \, + \, t \, a \, \xi) \cdot \xi \, dt \, d\xi \\ &+& \, a \, \nabla M_{B}\left( \tilde{u}_{1,m}(\cdot) \, \int_{0}^{1} \nabla \tau(z_{m} \, + \, t \, a \, \cdot)  \cdot  \, dt \right)(\eta) \\ &-& \, \omega^{2} \, \mu \, a^{3} \, \int_{B} \Gamma(z_{m} + a \eta,z_{m}+a\xi) \cdot \tilde{u}_{1,m}(\xi) \, (\epsilon_{0}(z_{m} \, + \, a \, \xi)-\epsilon_{p}(\omega)) \, d\xi, \qquad \eta \in B,
\end{eqnarray*}
where $\tilde{u}_{k,j}(\eta) \, := \, u_{k}(z_{j} \, + \, a \, \eta)$, for $k \, = \, 0 \,, \,1$ and $1 \leq j \leq \aleph$. For the above equation, we successively take  the inverse of the operator $\left[  I \, - \, \tau(z_{m}) \, \nabla M_{B}\right]$, and then the $\left\Vert \cdot \right\Vert_{\mathbb{L}^{2}(B)}$-norm on the both sides, by using the fact that $\left\Vert \nabla M_{B} \right\Vert_{\mathcal{L}\left( \mathbb{L}^{2}(B); \mathbb{L}^{2}(B) \right)} \, = \, 1$  and 
\begin{equation*}
   \left\Vert  \left[  I \, - \, \tau(z_{m}) \, \nabla M_{B}\right]^{-1} \right\Vert_{\mathcal{L}\left( \mathbb{L}^{2}(B); \mathbb{L}^{2}(B) \right)} \, = \, \mathcal{O}\left( a^{-h} \right), 
\end{equation*}
see for instance \cite[Subsection 4.1]{AG-MS-Maxwell}, to obtain  
\begin{eqnarray*}
\left\Vert \tilde{u}_{1,m} \right\Vert_{\mathbb{L}^{2}(B)}  \, & \lesssim & \, a^{-h} \, \Bigg[ \left\Vert \tilde{u}_{0,m} \right\Vert_{\mathbb{L}^{2}(B)} \, +  \, a^{3} \, \sum_{j=1 \atop j \neq m}^{\aleph}  \left\vert  G_{k}(z_{m},z_{j}) \right\vert \, \left\Vert \tilde{u}_{1,j} \right\Vert_{\mathbb{L}^{2}(B)}  \\ &+& \, a^{4} \, \sum_{j=1 \atop j \neq m}^{\aleph} \left\vert \nabla G_{k}(z_{m}, z_{j}) \right\vert  \, \left\Vert \tilde{u}_{1,j} \right\Vert_{\mathbb{L}^{2}(B)}  \,
%&+& \,  a^{4} \, \sum_{j=1 \atop j \neq m}^{\aleph}  \left\vert \nabla G_{k}(z_{m}, z_{j}) \right\vert \, \left\Vert \tilde{u}_{1,j} \right\Vert_{\mathbb{L}^{2}(B)} \\
% &-& \, \omega^{2} \, \mu \, a^{3} \, \sum_{j=1 \atop j \neq m}^{\aleph} \int_{B} G_{k}(z_{m} \, + \, a \, \eta , z_{j} \, + \, a \, \xi ) \cdot \tilde{u}_{1,j}(\xi) \, \int_{0}^{1} \nabla \epsilon_{0}(z_{m} \, + \, t \, a \, \xi) \cdot \xi \, dt \, d\xi \\ 
 + \, a \, \left\Vert \tilde{u}_{1,m} \right\Vert_{\mathbb{L}^{2}(B)} \\ &+&  \, a^{3} \, \left\Vert  \int_{B} \Gamma(z_{m} + a \cdot , z_{m}+a\xi) \cdot \tilde{u}_{1,m}(\xi) \, (\epsilon_{0}(z_{m} \, + \, a \, \xi)-\epsilon_{p}(\omega)) \, d\xi \right\Vert_{\mathbb{L}^{2}(B)}\Bigg]. 
\end{eqnarray*}
In addition, knowing that $h \, < \, 1$ and 
\begin{equation*}
    \Gamma(z_{m} + a \cdot , z_{m}+a\xi) \, \sim \, \Phi_{0}^{2}(z_{m} + a \cdot , z_{m}+a\xi) \, = \, a^{-2} \, \frac{1}{\left\vert \cdot \, - \, \xi \right\vert^{2}}, 
\end{equation*}
from singularity analysis point of view, see for instance \cite[Theorem 2.1]{AG-MS-Maxwell}, we deduce that 
\begin{eqnarray*}
\left\Vert \tilde{u}_{1,m} \right\Vert_{\mathbb{L}^{2}(B)}  \, & \lesssim & \, a^{-h} \, \Bigg[ \left\Vert \tilde{u}_{0,m} \right\Vert_{\mathbb{L}^{2}(B)} \, +  \, a^{3} \, \sum_{j=1 \atop j \neq m}^{\aleph}  \left\vert  G_{k}(z_{m},z_{j}) \right\vert \, \left\Vert \tilde{u}_{1,j} \right\Vert_{\mathbb{L}^{2}(B)}  \\ &+& \, a^{4} \, \sum_{j=1 \atop j \neq m}^{\aleph} \left\vert \nabla G_{k}(z_{m}, z_{j}) \right\vert  \, \left\Vert \tilde{u}_{1,j} \right\Vert_{\mathbb{L}^{2}(B)} \,  +  \, a \, \left\Vert \tilde{u}_{1,m}  \right\Vert_{\mathbb{L}^{2}(B)}\Bigg]. 
\end{eqnarray*}
Moreover, using the fact that 
\begin{equation*}
    \left\vert  G_{k}(z_{m},z_{j}) \right\vert \, \sim  \, \frac{1}{d^{3}_{mj}} \quad \text{and} \quad \left\vert  \nabla G_{k}(z_{m},z_{j}) \right\vert \, \sim  \, \frac{1}{d^{4}_{mj}}, 
\end{equation*}
we deduce that 
\begin{equation*}
\left\Vert \tilde{u}_{1,m} \right\Vert_{\mathbb{L}^{2}(B)}  \,  \lesssim  \, a^{-h} \, \left[ \left\Vert \tilde{u}_{0,m} \right\Vert_{\mathbb{L}^{2}(B)} \, +   \, \sum_{j=1 \atop j \neq m}^{\aleph} \left( \frac{a^{3}}{d^{3}_{mj}}  \,  + \,   \frac{a^{4}}{d^{4}_{mj}} \right) \, \left\Vert \tilde{u}_{1,j} \right\Vert_{\mathbb{L}^{2}(B)}  \right]. 
\end{equation*}
Besides, we have 
\begin{equation*}
    \left( \frac{a^{3}}{d^{3}_{mj}}  \,  + \,   \frac{a^{4}}{d^{4}_{mj}} \right) \, = \,     \frac{a^{3}}{d^{3}_{mj}} \, \left( 1  \,  + \,   \frac{a}{d_{mj}} \right) \
     \leq  \, \frac{a^{3}}{d^{3}_{mj}} \, \left( 1  \,  + \,   \frac{a}{d} \right) \,
     =  \, \frac{a^{3}}{d^{3}_{mj}} \, \left( 1  \,  + \,   a^{1-t} \right) \, = \, \mathcal{O}\left( \frac{a^{3}}{d^{3}_{mj}} \right),
\end{equation*}
where the last equality is due to the fact that $t \, < \, 1$. Then, 
\begin{equation*}
\left\Vert \tilde{u}_{1,m} \right\Vert_{\mathbb{L}^{2}(B)}  \,  \lesssim  \, a^{-h} \, \left[ \left\Vert \tilde{u}_{0,m} \right\Vert_{\mathbb{L}^{2}(B)} \, + \, a^{3} \, \sum_{j=1 \atop j \neq m}^{\aleph} \frac{1}{d^{3}_{mj}}  \, \left\Vert \tilde{u}_{1,j} \right\Vert_{\mathbb{L}^{2}(B)}  \right]. 
\end{equation*}
By taking the square and summing up with respect to the index $m$ on the both sides, we end up with the following estimation 
\begin{equation*}
\sum_{m=1}^{\aleph} \left\Vert \tilde{u}_{1,m} \right\Vert^{2}_{\mathbb{L}^{2}(B)}  \,  \lesssim  \, a^{-2h} \, \left[ \sum_{m=1}^{\aleph} \left\Vert \tilde{u}_{0,m} \right\Vert^{2}_{\mathbb{L}^{2}(B)} \, + \, a^{6} \, \sum_{m=1}^{\aleph} \sum_{j=1 \atop j \neq m}^{\aleph} \frac{1}{d^{6}_{mj}}  \, \sum_{j=1}^{\aleph}\left\Vert \tilde{u}_{1,j} \right\Vert^{2}_{\mathbb{L}^{2}(B)}  \right].
\end{equation*}
Since
\begin{equation*}
    \sum_{m=1}^{\aleph} \sum_{j=1 \atop j \neq m}^{\aleph} \frac{1}{d^{6}_{mj}} \, \leq \, d^{-6} \, \aleph^{2} \, = \, \mathcal{O}\left( a^{- 6 t - 2 s } \right).   
\end{equation*}
Then, 
\begin{equation*}
\sum_{m=1}^{\aleph} \left\Vert \tilde{u}_{1,m} \right\Vert^{2}_{\mathbb{L}^{2}(B)}  \,  \lesssim  \,   a^{-2h} \, \sum_{m=1}^{\aleph} \left\Vert \tilde{u}_{0,m} \right\Vert^{2}_{\mathbb{L}^{2}(B)} \, + \, a^{6-6t-2s-2h}   \, \sum_{j=1}^{\aleph}\left\Vert \tilde{u}_{1,j} \right\Vert^{2}_{\mathbb{L}^{2}(B)},  
\end{equation*}
and, under the condition 
\begin{equation*}
    3 \, - \, 3 \, t \, - \, s \, - \, h \, > \, 0, 
\end{equation*}
we deduce that 
\begin{eqnarray*}
\sum_{m=1}^{\aleph} \left\Vert \tilde{u}_{1,m} \right\Vert^{2}_{\mathbb{L}^{2}(B)}  \, & \lesssim & \,   a^{-2h} \, \sum_{m=1}^{\aleph} \left\Vert \tilde{u}_{0,m} \right\Vert^{2}_{\mathbb{L}^{2}(B)} \\
 \left\Vert \tilde{u}_{1} \right\Vert_{\mathbb{L}^{2}(B)}  \, & \lesssim & \,   a^{-h} \,  \left\Vert \tilde{u}_{0} \right\Vert_{\mathbb{L}^{2}(B)}.
\end{eqnarray*}
This implies, by scaling back to the domain $D$, 
\begin{equation*}
    \left\Vert u_{1} \right\Vert_{\mathbb{L}^{2}(D)}  \,  \lesssim  \,   a^{-h} \,  \left\Vert u_{0} \right\Vert_{\mathbb{L}^{2}(D)}.
\end{equation*}
This proves $(\ref{Equa1359})$ and completes the proof of Proposition \ref{Propapest}. 
\end{proof}
\subsection{Comprehending the formula $(\ref{Equa0805})$.}\label{SubSecIIEqua242}
In order to finish the proof of Proposition \ref{PropMER}, it is necessary to comprehend $(\ref{Equa0805})$. Recall $(\ref{Equa0805})$ that
\begin{equation*}
        \left( G_{k} \, - \, \Pi_{k} \right)(x,z) \, = \, \int_{\Omega} \Pi_{k}(x,y) \cdot \left( k^{2}(y) \, - \, k^{2} \right) \, G_{k}(y,z) \, dy. 
\end{equation*}
To give sense to the R.H.S of the above equation, we use the explicit expression of the Green dyadic function $\Pi_{k}(\cdot,\cdot)$ given by $(\ref{KernelP})$, to obtain 
\begin{eqnarray*}
    \int_{\Omega} \Pi_{k}(x,y) \cdot \left( k^{2}(y) \, - \, k^{2} \right) \, G_{k}(y,z) \, dy \, 
    & = & \, - \frac{1}{k^{2}} \, \underset{x}{\nabla} \int_{\Omega} \underset{y}{\nabla} \, \Phi_{k}(x,y)  \cdot \left( k^{2}(y) \, - \, k^{2} \right) \, G_{k}(y,z) \, dy \\ &+& \, \int_{\Omega}   \Phi_{k}(x,y) \,  \left( k^{2}(y) \, - \, k^{2} \right) \, G_{k}(y,z) \, dy,
\end{eqnarray*}
which, by using the definition of the Magnetization operator and the Newtonian operator, see $(\ref{DefNDefMk})$, can be rewritten as
\begin{eqnarray}\label{G-P}
\nonumber
    \int_{\Omega} \Pi_{k}(x,y) \cdot \left( k^{2}(y) \, - \, k^{2} \right) \, G_{k}(y,z) \, dy \, &=& \, - \, \frac{1}{k^{2}} \, \underset{x}{\nabla} M_{\Omega}^{k}\left(  \left( k^{2}(\cdot) \, - \, k^{2} \right) \, G_{k}(\cdot,z) \right)(x) \\ &+& N_{\Omega}^{k}\left(  \left( k^{2}(\cdot) \, - \, k^{2} \right) \, G_{k}(\cdot,z) \right)(x).
\end{eqnarray}
Now, for an arbitrary fixed point $z$, write
\begin{equation}\label{G=P+(G-P)}
    G_{k}(\cdot, z) \,= \, \Pi_{k}(\cdot, z) \, + \, \left(G_{k} - \Pi_{k} \right)(\cdot, z),
\end{equation}
in $\Omega$ with $\left(G_{k} - \Pi_{k} \right)(\cdot, z)$ being a regular function, see for instance \cite[Theorem 2.1]{AG-MS-Maxwell}. More precisely, based on the proof of \cite[Theorem 2.1]{AG-MS-Maxwell}, we know that  
\begin{equation}\label{IA1}
    \left(G_{k} - \Pi_{k} \right)(\cdot, z) \, \in \, \mathbb{L}^{\frac{3}{2}-\delta}(\Omega), \quad \text{with} \; \delta \, \rightarrow \,  0^{+}. 
\end{equation}
Then, using $(\ref{G=P+(G-P)})$, the equation $(\ref{G-P})$ takes the following form,
\begin{eqnarray}\label{LHSEq1}
\nonumber
    \int_{\Omega} \Pi_{k}(x,y) \cdot \left( k^{2}(y) \, - \, k^{2} \right) \, G_{k}(y,z) \, dy  \, 
    &=& \, - \, \frac{1}{k^{2}} \, \underset{x}{\nabla} M_{\Omega}^{k}\left(  \left( k^{2}(\cdot) \, - \, k^{2} \right) \, \Pi_{k}(\cdot,z) \right)(x) \\ \nonumber &+& N_{\Omega}^{k}\left(  \left( k^{2}(\cdot) \, - \, k^{2} \right) \, \Pi_{k}(\cdot,z) \right)(x) \\ \nonumber
    &-& \frac{1}{k^{2}} \, \underset{x}{\nabla} M_{\Omega}^{k}\left(  \left( k^{2}(\cdot) \, - \, k^{2} \right) \, \left(G_{k}-\Pi_{k}\right)(\cdot,z) \right)(x) \\ &+& N_{\Omega}^{k}\left(  \left( k^{2}(\cdot) \, - \, k^{2} \right) \, \left(G_{k}-\Pi_{k}\right)(\cdot,z) \right)(x).
\end{eqnarray}
The third term and the fourth term on the R.H.S are well defined, as we know that $\left(G_{k} - \Pi_{k} \right)(\cdot, z)$ is a regular function and both the Magnetization operator and the Newtonian operator are bounded from $\mathbb{L}^{p}(\mathbb{R}^{3})$ to $\mathbb{L}^{p}(\mathbb{R}^{3})$, with $p > 1$. Consequently, based on $(\ref{IA1})$, we deduce that 
\begin{equation}\label{T0z}
    T_{0}(x,z) \, := - \, \frac{1}{k^{2}} \, \underset{x}{\nabla} M_{\Omega}^{k}\left(  \left( k^{2}(\cdot) \, - \, k^{2} \right) \, \left(G_{k}-\Pi_{k}\right)(\cdot,z) \right)(x) + N_{\Omega}^{k}\left(  \left( k^{2}(\cdot) \, - \, k^{2} \right) \, \left(G_{k}-\Pi_{k}\right)(\cdot,z) \right)(x) \; \in \; \mathbb{L}^{\frac{3}{2}-\delta}(\Omega).
\end{equation}
Hence, to give sense of the L.H.S of $(\ref{LHSEq1})$, it is enough to  study the first term and the second term on the R.H.S accordingly. To accomplish this, we set 
\begin{eqnarray*}
    T_{1}(x,z) &:=& \underset{x}{\nabla} M_{\Omega}^{k}\left(  \left( k^{2}(\cdot) \, - \, k^{2} \right) \, \Pi_{k}(\cdot,z) \right)(x) \\
    &\overset{(\ref{KernelP})}{=}& \frac{1}{k^{2}} \underset{x}{\nabla} M_{\Omega}^{k}\left(  \left( k^{2}(\cdot) \, - \, k^{2} \right) \, \nabla \nabla \Phi_{k}(\cdot,z) \right)(x) + \underset{x}{\nabla} M_{\Omega}^{k}\left(  \left( k^{2}(\cdot) \, - \, k^{2} \right) \, \Phi_{k}(\cdot,z) I_{3} \right)(x). 
\end{eqnarray*}
Observe that the second term on the R.H.S is well defined as $\Phi_{k}(\cdot,z) I_{3} \in \mathbb{L}^{3-\delta}(\Omega)$ and $\nabla M_{\Omega}^{k}\left( \cdot \right)$ is bounded from $\mathbb{L}^{p}(\mathbb{R}^{3})$ to $\mathbb{L}^{p}(\mathbb{R}^{3})$, with $p > 1$. Then,
\begin{equation}\label{T12L3}
    T_{1,2}(\cdot,z) \; := \; \underset{x}{\nabla} M_{\Omega}^{k}\left(  \left( k^{2}(\cdot) \, - \, k^{2} \right) \, \Phi_{k}(\cdot,z) I_{3} \right) \, \in \, \mathbb{L}^{3-\delta}(\Omega),
\end{equation} 
which implies, 
\begin{eqnarray*}
    T_{1}(x,z) &=&  \frac{1}{k^{2}} \underset{x}{\nabla} M_{\Omega}^{k}\left(  \left( k^{2}(\cdot) \, - \, k^{2} \right) \, \nabla \nabla \Phi_{k}(\cdot,z) \right)(x) \, + \,  T_{1,2}(x,z)  \\
    &=& - \, \frac{1}{k^{2}} \underset{x}{\nabla} M_{\Omega}^{k}\left(  \left( k^{2}(\cdot) \, - \, k^{2} \right) \, \underset{z}{\nabla} \nabla \Phi_{k}(\cdot,z) \right)(x) \, + \, T_{1,2}(x,z)  \\
    &=& - \, \frac{1}{k^{2}} \underset{z}{\nabla} \, \left( \underset{x}{\nabla}  M_{\Omega}^{k}\left(  \left( k^{2}(\cdot) \, - \, k^{2} \right) \,  \nabla \Phi_{k}(\cdot,z) \right)(x) \right) \, + \, T_{1,2}(x,z).
\end{eqnarray*}
Next, we need to analyze the regularity of the term 
\begin{equation}\label{T11xz}
    T_{1,1}(x,z) \, := \, -\frac{1}{k^{2}} \, \underset{z}{\nabla} \left( \underset{x}{\nabla}  M_{\Omega}^{k}\left(  \left( k^{2}(\cdot) \, - \, k^{2} \right) \,  \nabla \Phi_{k}(\cdot,z) \right)(x) \right). 
\end{equation}
In order to avoid lengthy computation, we only examine the  reduced formula given by 
\begin{eqnarray}\label{T11r}
\nonumber
    T_{1,1,r}(x,z) \, &=&  \, \underset{z}{\nabla} \left( \underset{x}{\nabla}  M_{\Omega}^{k}\left( \nabla \Phi_{k}(\cdot,z) \right)(x) \right) \\ \nonumber
    & \overset{(\ref{DefNDefM})}{=} & \, \underset{z}{\nabla} \, \underset{x}{\nabla} \int_{\Omega} \, \underset{y}{\nabla} \Phi_{k}(x,y) \cdot \underset{y}{\nabla} \Phi_{k}(y,z) \, dy \\
    &=& \underset{x}{\nabla} \left( \underset{z}{\nabla} \left( \underset{z}{\nabla} \cdot \left( \underset{x}{\nabla} \cdot \left( N^{k}_{\Omega}\left( \Phi_{k}(\cdot,z) \, I_{3} \right)(x) \right) \right) \right) \right), 
\end{eqnarray}
where $N^{k}(\cdot)$ is the Newtonian operator defined by $(\ref{DefNDefMk})$. By using the Helmholtz decomposition $(\ref{HelmholtzDecomposition})$, we have
%\, + \, \sum_{n \in \mathbb{N}} \langle \Phi_{k}(\cdot,z) \, I_{3}; e_{n}^{(2)} \rangle_{\mathbb{L}^{2}(\Omega)} \otimes e_n^{(2)}(\cdot) \\ \nonumber &+& \, \sum_{n \in \mathbb{N}} \langle \Phi_{k}(\cdot,z) \, I_{3}; e_{n}^{(3)} \rangle_{\mathbb{L}^{2}(\Omega)} \otimes e_n^{(3)}(\cdot) 
%\, + \, \sum_{n \in \mathbb{N}} \, N^{k}_{\Omega}\left(e_{n}^{(2)}\right)(z)  \otimes e_n^{(2)}(\cdot) \, + \, \sum_{n \in \mathbb{N}} N^{k}_{\Omega}\left(e_{n}^{(3)}\right)(z)  \otimes e_n^{(3)}(\cdot)
\begin{equation*}
    \Phi_{k}(\cdot,z) \, I_{3} \, = \,  \sum_{j=1}^{3} \sum_{n \in \mathbb{N}} \langle \Phi_{k}(\cdot,z) \, I_{3}, e_{n}^{(j)} \rangle_{\mathbb{L}^{2}(\Omega)} \otimes e_n^{(j)}(\cdot) 
    \overset{(\ref{DefNDefMk})}{=} \, \sum_{j=1}^{3} \, \sum_{n \in \mathbb{N}} N^{k}_{\Omega}\left(e_{n}^{(j)}\right)(z) \otimes e_n^{(j)}(\cdot).
\end{equation*}
Hence, 
%\\ &+& \, \sum_{n \in \mathbb{N}} \, N^{k}_{\Omega}\left(e_{n}^{(2)}\right)(z)  \otimes N^{k}_{\Omega}\left(e_n^{(2)}\right)(x) \, + \, \sum_{n \in \mathbb{N}} N^{k}_{\Omega}\left(e_{n}^{(3)}\right)(z)  \otimes N^{k}_{\Omega}\left(e_n^{(3)}\right)(x)
\begin{equation}\label{PhikDecomposition}
    N^{k}_{\Omega}\left( \Phi_{k}(\cdot,z) \, I_{3} \right)(x) \, = \, \sum_{j=1}^{3} \, \sum_{n \in \mathbb{N}} N^{k}_{\Omega}\left(e_{n}^{(j)}\right)(z) \otimes N^{k}_{\Omega}\left(e_n^{(j)}\right)(x). 
\end{equation}
Then, by plugging $(\ref{PhikDecomposition})$ into $(\ref{T11r})$, we deduce 
\begin{eqnarray*}
    T_{1,1,r}(x,z) \, &=& \, \sum_{j=1}^{3} \, \sum_{n \in \mathbb{N}} \, \underset{x}{\nabla} \left( \underset{z}{\nabla} \left( \underset{z}{\nabla} \cdot \left( \underset{x}{\nabla} \cdot \left(  N^{k}_{\Omega}\left(e_{n}^{(j)}\right)(z)  \otimes N^{k}_{\Omega}\left(e_n^{(j)}\right)(x) \right) \right) \right) \right) \\
    &=& \, \sum_{j=1}^{3} \, \sum_{n \in \mathbb{N}} \, \underset{x}{\nabla} \left( \underset{z}{\nabla} \left( \underset{z}{\nabla} \cdot  \left(  N^{k}_{\Omega}\left(e_{n}^{(j)}\right)(z) \right)  \underset{x}{\nabla} \cdot \left( N^{k}_{\Omega}\left(e_n^{(j)}\right)(x) \right) \right) \right), 
\end{eqnarray*}
and, using the fact that for an arbitrary vector field $F$, there holds
\begin{equation*}
    \nabla M^{k}_{\Omega}\left( F \right) \, = \, - \, \nabla \nabla \cdot N^{k}_{\Omega}\left( F \right),
\end{equation*}
we derive that 
\begin{equation*}
        T_{1,1,r}(x,z) \, = \, \sum_{j=1}^{3} \, \sum_{n \in \mathbb{N}} \, \nabla M^{k}_{\Omega}\left(e_{n}^{(j)}\right)(z)  \otimes \nabla M^{k}_{\Omega}\left(e_n^{(j)}\right)(x).
\end{equation*}
Furthermore, knowing that $\nabla M^{k}_{\Omega}$ restricted to $\mathbb{H}_{0}(\div = 0)$ is a vanishing operator, we obtain 
\begin{equation*}
        T_{1,1,r}(x,z) \, = \, \sum_{j=2}^{3} \, \sum_{n \in \mathbb{N}} \, \nabla M^{k}_{\Omega}\left(e_{n}^{(j)}\right)(z)  \otimes \nabla M^{k}_{\Omega}\left(e_n^{(j)}\right)(x).
\end{equation*}
By recalling that $e_n^{(2,3)}(\cdot) \, \in \, \mathbb{L}^{2}(\Omega)$ and $\nabla M^{k}_{\Omega}(\cdot)$ is a bounded operator from $\mathbb{L}^{2}(\Omega)$ to $\mathbb{L}^{2}(\Omega)$ we deduce that, for $z$ fixed in $\Omega$, 
\begin{equation*}
    T_{1,1,r}(\cdot, z) \, \in \, \mathbb{L}^{2}(\Omega).  
\end{equation*}
Hence, for $(\ref{T11xz})$, 
\begin{equation}\label{T11L2}
    T_{1,1}(\cdot, z) \, \in \, \mathbb{L}^{2}(\Omega).  
\end{equation}
Then, by gathering $(\ref{T12L3})$ and $(\ref{T11L2})$, we deduce that 
\begin{equation}\label{T1z}
    T_{1}(\cdot, z) \, \in \, \mathbb{L}^{2}(\Omega).
\end{equation}
In a similar way, for $z$ fixed in $\Omega$, we set 
\begin{eqnarray*}
    T_{2}(x,z) \, &:=& \, N_{\Omega}^{k}\left(  \left( k^{2}(\cdot) \, - \, k^{2} \right) \, \Pi_{k}(\cdot,z) \right)(x) \\ & \overset{(\ref{KernelP})}{=} & \, \frac{1}{k^{2}} \, N_{\Omega}^{k}\left(  \left( k^{2}(\cdot) \, - \, k^{2} \right) \, \nabla \nabla \Phi_{k}(\cdot,z) \right)(x) \, + \, N_{\Omega}^{k}\left(  \left( k^{2}(\cdot) \, - \, k^{2} \right) \, \Phi_{k}(\cdot,z) \, I_{3} \right)(x). 
\end{eqnarray*}
Since the second term on the R.H.S fulfills $\Phi_{k}(\cdot,z) I_{3} \in \mathbb{L}^{3-\delta}(\Omega)$, we deduce that 
\begin{equation*}
    N_{\Omega}^{k}\left(  \left( k^{2}(\cdot) \, - \, k^{2} \right) \, \Phi_{k}(\cdot,z) \, I_{3} \right)(x) \, \in \, \mathbb{W}^{2,3-\delta}(\Omega) \subset \mathbb{L}^{3-\delta}(\Omega).
\end{equation*}
Hence, in the sequel, to write short formulas, we denote it by $T_{2,2}(x,z)$. Then,   
\begin{eqnarray*}
    T_{2}(x,z) \, &=& \, \frac{1}{k^{2}} \, N_{\Omega}^{k}\left(  \left( k^{2}(\cdot) \, - \, k^{2} \right) \, \nabla \nabla \Phi_{k}(\cdot,z) \right)(x) \, + \, T_{2,2}(x,z) \\
    &=& \, \frac{1}{k^{2}} \, \underset{x}{\nabla} \, \underset{x}{\nabla} N_{\Omega}^{k}\left(  \left( k^{2}(\cdot) \, - \, k^{2} \right) \, \Phi_{k}(\cdot,z) \right)(x) \, + \, T_{2,2}(x,z).
\end{eqnarray*}
Now, using the Calderon-Zygmund inequality, see \cite[Theorem 9.9]{gilbarg2001elliptic}, we deduce that 
\begin{equation*}
    \underset{x}{\nabla} \, \underset{x}{\nabla} N_{\Omega}^{k}\left(  \left( k^{2}(\cdot) \, - \, k^{2} \right) \, \Phi_{k}(\cdot,z) \right)(x) \, \in \, \mathbb{L}^{3 - \delta}(\Omega).
\end{equation*}
This implies,  
\begin{equation}\label{T2z}
    T_{2}(\cdot, z) \, \in \, \mathbb{L}^{3 - \delta}(\Omega).
\end{equation}
Gathering $(\ref{T0z}), \, (\ref{T1z})$ and $(\ref{T2z})$ allows us to deduce  
\begin{equation*}
    \int_{\Omega} \Pi_{k}(\cdot,y) \cdot \left( k^{2}(y) \, - \, k^{2} \right) \, G_{k}(y,z) \, dy \, \in  \, \mathbb{L}^{\frac{3}{2}-\delta}(\Omega),
\end{equation*}
and proves the well-posed character of  $(\ref{Equa0805})$. 
\subsection{Justification of \ref{GT=G}}\label{SubsectionGT=G}
Let $E$ and $F$ be two compactly supported smooth vector fields. We consider the solutions to the two following problems:
\begin{eqnarray}
    \nabla \times \nabla \times \left( V^{E}(x) \right) \, - \, k^{2}(x) \, V^{E}(x) \, &=&  \, E(x), \quad \text{in} \quad \mathbb{R}^{3}, \label{Eq1244-1} \\ 
    \nabla \times \nabla \times \left( V^{F}(x) \right) \, - \, k^{2}(x) \, V^{F}(x) \, &=&  \, F(x), \quad \text{in} \quad \mathbb{R}^{3}, \label{Eq1244} 
\end{eqnarray}
where $V^E$ and $V^F$ satifsfy the Silver-M\"{u}ller radiation conditions. Based on $(\ref{Eq0401})$, the solution $V^{E}(\cdot)$ can be given by 
\begin{equation*}
    V^{E}(x) \, := \, \left( \, G_{k} \, \star \, E \, \right)(x), \quad x \in \mathbb{R}^{3}, 
\end{equation*}
where the convolution operator should be understood in the following sense,  
\begin{equation}\label{VEGE}
        V^{E}(x) \, = \, \int_{\mathbb{R}^{3}} G_{k}(x,y) \cdot E(y) \, dy.
\end{equation}
The formula above can be comprehended as explained in Section \ref{SecPrfThm}, formula $(\ref{Int(x)})$. In the same manner, we have
\begin{equation}\label{Equa1328}
    V^{F}(x) \, := \, \left( \, G_{k} \, \star \, F \, \right)(x) \, = \, \int_{\mathbb{R}^{3}} G_{k}(x,y) \cdot F(y) \, dy, \quad x \in \mathbb{R}^{3}.
\end{equation}
Since the vector fields $E$ and $F$ are compactly supported, then we can replace the integral set by a ball of center zero and radius $R$, i.e. $B(0,R)$, with $R$ large enough such that $\Omega \subset B(0,R)$. In $\mathbb{L}^{2}(B(0,R))$, by taking the inner product for $(\ref{Eq1244})$ with $V^E$ and for $(\ref{Eq1244-1})$ with $V^F$, we can obtain that
\begin{equation*}
\langle E, V^{F} \rangle_{\mathbb{L}^{2}(B(0,R))} \, - \, \langle V^E, F \rangle_{\mathbb{L}^{2}(B(0,R))} 
   \, = \, \int_{B(0,R)} \underset{x}{\nabla} \cdot \left[ \left( \nabla \times \left( V^{E} \right) \right) \times V^{F} \, + \, V^{E} \times  \left( \nabla \times\left(V^{F}\right)\right) \right] \, dx. 
\end{equation*}
Moreover, for the same reason that $E$ and $F$ are of compact support, we can replace the inner product $\langle \cdot ; \cdot \rangle_{\mathbb{L}^{2}(B(0,R))}$, appearing on the L.H.S, by $\langle \cdot ; \cdot \rangle_{\mathbb{L}^{2}\left(\mathbb{R}^{3}\right)}$. Then, 
\begin{equation*}
\langle E, V^{F} \rangle_{\mathbb{L}^{2}\left(\mathbb{R}^{3}\right)} \, - \, \langle V^E, F \rangle_{\mathbb{L}^{2}\left(\mathbb{R}^{3}\right)} 
   \, = \, \int_{B(0,R)} \underset{x}{\nabla} \cdot \left[ \left( \nabla \times \left( V^{E} \right) \right) \times V^{F} \, + \, V^{E} \times  \left( \nabla \times\left( V^{F} \right)\right) \right] \, dx.  
\end{equation*}
In addition, since the L.H.S is independent on the parameter $R$, by taking the limit as $R\rightarrow\infty$, we obtain 
\begin{equation}\label{VFE=FVE+J}
\langle E, V^{F} \rangle_{\mathbb{L}^{2}\left(\mathbb{R}^{3}\right)} \, - \, \langle V^E, F \rangle_{\mathbb{L}^{2}\left(\mathbb{R}^{3}\right)} 
   \, = \, J,
\end{equation}
where
\begin{equation*}
         J  \, := \, \lim_{R \rightarrow + \infty} \int_{B(0,R)} \underset{x}{\nabla} \cdot \left[ \left( \nabla \times \left( V^{E} \right) \right) \times V^{F} \, + \, V^{E} \times  \left( \nabla \times\left( V^{F} \right)\right) \right] \, dx. 
\end{equation*}
Now, we show that $J$ is vanishing. By using the Divergence theorem to rewrite $J$ as,
\begin{eqnarray*}
      J \, &=& \, \lim_{R \rightarrow + \infty} \int_{\partial B(0,R)} \left(\nabla \times \left( V^{E} \right) \cdot   \left( V^{F} \times \nu \right) \, - \, \nabla \times \left( V^{F} \right) \cdot \left( V^{E} \times \nu \right) \right)\, d\sigma \\
      &=& \, \lim_{R \rightarrow + \infty} \int_{\partial B(0,R)} \left(\nabla \times \left( V^{E} \right) \cdot   \left( V^{F} \times \nu \, - \, \nabla \times \left(V^{F} \right) \right) \, + \, \nabla \times \left(V^{E}\right) \cdot \nabla \times \left( V^{F} \right) \right) \, d\sigma \\
      &-& \, \lim_{R \rightarrow + \infty} \int_{\partial B(0,R)}  \, \left( \nabla \times \left( V^{F}\right) \cdot   \left( V^{E} \times \nu \, - \, \nabla \times \left(V^{E} \right) \right) \, + \, \nabla \times \left( V^{F} \right) \cdot \nabla \times \left(V^{E} \right) \right) \, d\sigma,
\end{eqnarray*}
which can be simplified as
\begin{eqnarray}\label{J=J1-J2}
\nonumber
      J \, &=& \, \lim_{R \rightarrow + \infty} \int_{\partial B(0,R)} \nabla \times \left( V^{E} \right) \cdot   \left( V^{F} \times \nu \, - \, \nabla \times \left( V^{F} \right) \right)  \, d\sigma \\
      &-& \, \lim_{R \rightarrow + \infty} \int_{\partial B(0,R)}  \, \nabla \times (V^{F}) \cdot   \left( V^{E} \times \nu \, - \, \nabla \times \left(V^{E} \right) \right)   \, d\sigma.
\end{eqnarray}
We set and estimate the second term on the R.H.S. as
\begin{eqnarray}\label{EstJ1} 
\nonumber
    J_{2} \, &:=& \, \lim_{R \rightarrow + \infty} \int_{\partial B(0,R)} \nabla \times V^{F} \cdot   \left( V^{E} \times \nu \, - \, \nabla \times V^{E}  \right) \, d\sigma \\
    \left\vert J_{2} \right\vert \, & \leq & \, \lim_{R \rightarrow + \infty}  \left\Vert \nabla \times V^{F} \right\Vert_{\mathbb{L}^{2}\left( \partial B(0,R) \right)} \, \left\Vert V^{E} \times \nu \, - \, \nabla \times V^{E}  \right\Vert_{\mathbb{L}^{2}\left( \partial B(0,R) \right)}.
\end{eqnarray}
Since $V^{E}$ is a radiating solution to the Maxwell equation, the following Silver-M\"{u}ller radiation condition holds, 
\begin{equation}\label{SMRC}
    \lim_{\left\vert x \right\vert \rightarrow + \infty} \left\vert x \right\vert \, \left( V^{E} \times \nu \, - \, \nabla \times V^{E} \right) \, = \, 0,
\end{equation}
see \cite[Definition 6.6]{colton2019inverse}. The formula $(\ref{SMRC})$, for $\left\vert x \right\vert \gg 1$, implies 
\begin{equation}\label{SMRC+}
     \left\vert V^{E} \times \nu \, - \, \nabla \times V^{E} \right\vert \, \lesssim \,  \frac{1}{\left\vert x \right\vert^{1+\alpha}}, \quad \alpha \in \mathbb{R}^{+}, 
\end{equation}
hence, for $R \gg 1$, we deduce 
\begin{eqnarray*}
    \left\Vert V^{E} \times \nu \, - \, \nabla \times V^{E}  \right\Vert^{2}_{\mathbb{L}^{2}\left( \partial B(0,R) \right)} \, &:=& \, \int_{\partial B(0,R)} \left\vert V^{E} \times \nu \, - \, \nabla \times V^{E}  \right\vert^{2}(x) \, d\sigma(x) \\
    & \overset{(\ref{SMRC+})}{\lesssim} & \, \int_{\partial B(0,R)} \frac{1}{\left\vert x \right\vert^{2(1+\alpha)}}  \, d\sigma(x) \, = \,  \frac{4 \, \pi}{\left\vert R \right\vert^{2 \, \alpha}}. 
\end{eqnarray*}
Using the above estimation, the inequality $(\ref{EstJ1})$ becomes, 
\begin{equation}\label{J1CurlV}
    \left\vert J_{2} \right\vert \,  \lesssim  \, \lim_{R \rightarrow + \infty}  \left\Vert \nabla \times V^{F} \right\Vert_{\mathbb{L}^{2}\left( \partial B(0,R) \right)} \, \frac{1}{R^{\alpha}}.
\end{equation}
Now, we are in a position to estimate $\left\Vert \nabla \times V^{F} \right\Vert_{\mathbb{L}^{2}\left( \partial B(0,R) \right)}$. By taking the Curl operator on the both sides of $(\ref{Equa1328})$, there holds
\begin{equation*}
    \underset{x}{\nabla} \times \left(V^{F}(x)\right) \, = \, \underset{x}{\nabla} \times \int_{\Omega} G_{k}(x,y) \cdot F(y) \, dy 
    = \, \underset{x}{\nabla} \times \int_{\Omega} \Pi_{k}(x,y) \cdot F(y) \, dy \, + \, \underset{x}{\nabla} \times \int_{\Omega} \left( G_{k} \, - \, \Pi_{k} \right)(x,y) \cdot F(y) \, dy. 
\end{equation*}
Now, by using the $\Pi_{k}(\cdot,\cdot)$'s expression given by $(\ref{KernelP})$, the definition of both the Magnetization and the Newtonian operators given by $(\ref{DefNDefMk})$ as well as the expression of $\left( G_{k} \, - \, \Pi_{k} \right)(\cdot,\cdot)$ given by $(\ref{Equa0805})$, we can rewrite the above equation as 
\begin{equation*}
   % Curl\left(V^{F}\right)(x) \, &=& \, \underset{x}{\nabla} \times \left( - \frac{1}{k^{2}} \, \nabla M^{k}_{\Omega}\left(F\right)(x) \, + \, N^{k}_{\Omega}\left(F\right)(x)  \right) \, + \, \underset{x}{\nabla} \times \int_{\Omega} \left( G_{k} \, - \, \Pi_{k} \right)(x,y) \cdot F(y) \, dy \\
    \nabla \times \left( V^{F} \right) \, = \, \nabla \times \left(  \, N^{k}_{\Omega}\left(F\right) \right) \, + \, \nabla \times \left( N^{k}_{\Omega}\left( \left(k^{2}(\cdot) \, - \, k^{2} \right) \, V^{F}(\cdot) \right) \right). 
\end{equation*}
Now, using the fact that\footnote{We recall that $SL^{k}_{\partial \Omega}$ is the Single-Layer operator defined by 
\begin{equation*}
    SL^{k}_{\partial \Omega}\left( f \right)(x) \, := \, \int_{\partial \Omega} \phi_{k}(x,y) \, f(y) \, d\sigma(y), \quad x \in \mathbb{R}^{3},
\end{equation*}
} 
\begin{equation*}
    \nabla \times N^{k}_{\Omega}\left( E \right) \, = \,  N^{k}_{\Omega}\left( \nabla \times E \right) \, - \, SL^{k}_{\partial \Omega}\left( \nu \times E \right), 
\end{equation*}
we deduce 
\begin{equation*}
    \nabla \times \left( V^{F}\right) \, = \,   N^{k}_{\Omega}\left( \nabla \times F\right) \, - \, SL^{k}_{\partial \Omega}\left( \nu \times F \right) \, + \,  N^{k}_{\Omega}\left( \nabla \times \left( \left(k^{2}(\cdot) \, - \, k^{2} \right) \, V^{F}(\cdot) \right) \right) \, - \,  SL^{k}_{\partial \Omega}\left(  \left(k^{2}(\cdot) \, - \, k^{2} \right) \nu \times V^{F}(\cdot) \right). 
\end{equation*}
Then, by taking $\left\Vert \cdot \right\Vert_{\mathbb{L}^{2}(\partial B(0,R))}$-norm on the both sides of the above equation, we obtain 
%Based on \cite[Theorem 2.1]{AG-MS-Maxwell}, from singularity analysis point of view, we have 
%\begin{equation*}
%    \left( G_{k} \, - \, \Pi_{k} \right)(x,y) \, \sim \, \Phi_{0}^{2}(x,y). 
%\end{equation*}
%Hence, 
%\begin{equation*}
%    Curl\left(V^{F}\right)(x) \,  \simeq  \,   N^{k}_{\Omega}\left( \nabla \times F\right)(x) \, - \, SL^{k}_{\partial \Omega}\left( \nu \times F \right)(x) \, + \, \underset{x}{\nabla} \times \int_{\Omega} \Phi_{0}^{2}(x,y) \, F(y) \, dy, 
%\end{equation*}
%which can be rewritten
%\begin{eqnarray*}    
%    Curl\left(V^{F}\right)(x) \, & \simeq  & \,   N^{k}_{\Omega}\left( \nabla \times F\right)(x) \, - \, SL^{k}_{\partial \Omega}\left( \nu \times F \right)(x) \\ &+& \,  \int_{\Omega} \Phi_{0}^{2}(x,y) \, \left( \nabla \times F \right)(y) \, dy \, - \, \int_{\partial \Omega} \Phi_{0}^{2}(x,y) \, \left( \nu \times F \right)(y) \, d\sigma(y).
%\end{eqnarray*}
%As $\left\vert x \right\vert \, \gg \, 1$, the first term on the R.H.S dominates the third term and the second term domaintes the fourth term. Thus,  
\begin{eqnarray*}
%    Curl\left(V^{F}\right)(x) \, & \simeq & \,   N^{k}_{\Omega}\left( \nabla \times F\right)(x) \, - \, SL^{k}_{\partial \Omega}\left( \nu \times F \right)(x) \\
    \left\Vert \nabla \times V^{F} \right\Vert^{2}_{\mathbb{L}^{2}(\partial B(0,R))} \, & \lesssim & \,  \left\Vert N^{k}_{\Omega}\left( \nabla \times F\right) \right\Vert^{2}_{\mathbb{L}^{2}(\partial B(0,R))} \, + \, \left\Vert SL^{k}_{\partial \Omega}\left( \nu \times F \right) \right\Vert^{2}_{\mathbb{L}^{2}(\partial B(0,R))} \\
     &+& \, \left\Vert N^{k}_{\Omega}\left( \nabla \times \left( \left(k^{2}(\cdot) \, - \, k^{2} \right) \, V^{F}(\cdot) \right) \right) \right\Vert^{2}_{\mathbb{L}^{2}(\partial B(0,R))} \\ &+& \, \left\Vert  SL^{k}_{\partial \Omega}\left(  \left(k^{2}(\cdot) \, - \, k^{2} \right) \nu \times V^{F}(\cdot) \right) \right\Vert^{2}_{\mathbb{L}^{2}(\partial B(0,R))} \\ 
    & \lesssim & \, \int_{\partial B(0,R)} \int_{\Omega} \left\vert \Phi_{0}(x,y) \right\vert^{2} \, dy \, d\sigma(x)  \;\; \rho_{1} + \, \int_{\partial B(0,R)} \int_{\partial \Omega} \left\vert \Phi_{0}(x,y) \right\vert^{2} \, d\sigma(y) \, d\sigma(x)  \; \; \rho_{2},
\end{eqnarray*}
where 
\begin{eqnarray*}
    \rho_{1} &:=& \left\Vert  \nabla \times F  \right\Vert^{2}_{\mathbb{L}^{2}(\Omega)} \, + \, \left\Vert  \nabla \times \left( \left(k^{2}(\cdot) - k^{2} \right) V^F(\cdot) \right)  \right\Vert^{2}_{\mathbb{L}^{2}(\Omega)}  \\
    \rho_{2} &:=& \left\Vert  \nu \times F  \right\Vert^{2}_{\mathbb{L}^{2}(\partial \Omega)} \, + \, \left\Vert   \left(k^{2}(\cdot) - k^{2} \right)  \nu \times V^F  \right\Vert^{2}_{\mathbb{L}^{2}(\partial\Omega)}
\end{eqnarray*}
As we have assumed the vector field $F$ to be smooth, hence $V^F$ will also be smooth and the following relation holds 
\begin{equation*}
   \max\left( \rho_{1} \, ;  \, \rho_{2} \right) \, = \, \mathcal{O}\left( 1 \right). 
\end{equation*}
This implies that
\begin{equation*}
    \left\Vert \nabla \times V^{F}  \right\Vert^{2}_{\mathbb{L}^{2}(\partial B(0,R))} \,
     \lesssim  \, \int_{\partial B(0,R)} \int_{\Omega} \left\vert \Phi_{0}(x,y) \right\vert^{2} \, dy \, d\sigma(x)  \, + \, \int_{\partial B(0,R)} \int_{\partial \Omega} \left\vert \Phi_{0}(x,y) \right\vert^{2} \, d\sigma(y) \, d\sigma(x) .
\end{equation*}
Since $\left\vert x \right\vert \gg 1$, we obtain 
\begin{equation*}
    \left\Vert \nabla \times V^{F} \right\Vert^{2}_{\mathbb{L}^{2}(\partial B(0,R))} \,
     \lesssim  \, \int_{\partial B(0,R)} \frac{1}{  \left\vert x \right\vert^{2}} \, d\sigma(x) \, = \, 4 \, \pi.
\end{equation*}
Therefore, in $(\ref{J1CurlV})$, we deduce that  
\begin{equation*}
    \left\vert J_{2} \right\vert \,  \lesssim \, \lim_{R \rightarrow + \infty}  \, \frac{1}{R^{\alpha}} \, = \, 0.
\end{equation*}
Similar arguments allow us to justify that 
\begin{equation*}
    J_{1} \, := \,  \lim_{R \rightarrow + \infty} \int_{\partial B(0,R)}  Curl(V^{E}) \cdot   \left( V^{F} \times \nu \, - \, Curl\left( V^{F} \right) \right)  \, d\sigma \, = \, 0.
\end{equation*}
Then, as $ J \, = \, J_{1} \, - \, J_{2}$, see $(\ref{J=J1-J2})$, we deduce that 
\begin{equation*}
    J \, = \, 0,
\end{equation*}
from
$(\ref{VFE=FVE+J})$, which further indicates that
\begin{equation*}
    \langle E, V^{F} \rangle_{\mathbb{L}^{2}(\mathbb{R}^{3})} \, = \, \langle V^E, F \rangle_{\mathbb{L}^{2}(\mathbb{R}^{3})}. 
\end{equation*}
By using $(\ref{VEGE})$ and $(\ref{Equa1328})$, we can know that
\begin{eqnarray*}
\int_{\mathbb{R}^{3}} E(x) \cdot \int_{\mathbb{R}^{3}} G_{k}(x,y) \cdot F(y) \, dy  \, dx \, &=& \, \int_{\mathbb{R}^{3}}  \int_{\mathbb{R}^{3}} G_{k}(x,y) \cdot E(y) \, dy \cdot F(x) dx  \\
& = & \, \int_{\mathbb{R}^{3}} \int_{\mathbb{R}^{3}}  G_{k}(x,y) \cdot E(y) \cdot F(x) dy \, dx \\
& = & \, \int_{\mathbb{R}^{3}} \int_{\mathbb{R}^{3}} E(y) \cdot G_{k}^{\top}(x,y) \cdot F(x)     \, dy \, dx \\
& = & \, \int_{\mathbb{R}^{3}} E(y) \cdot \int_{\mathbb{R}^{3}} G_{k}^{\top}(x,y) \cdot F(x) \, dx  \, dy  \\
& = & \, \int_{\mathbb{R}^{3}} E(x) \cdot \int_{\mathbb{R}^{3}} G_{k}^{\top}(y,x) \cdot F(y) \, dy      \, dx. 
\end{eqnarray*}
Since $E$ and $F$ are two arbitrary vector fields, we deduce that 
\begin{equation*}
    G_{k}(x,y) \, = \, G_{k}^{\top}(y,x), \quad x \neq y. 
\end{equation*}
This justify $(\ref{GT=G})$. 
\end{document}